\documentclass[11pt]{amsart}

\usepackage[all]{xy}
\usepackage{enumitem}
\usepackage{mathtools}
\usepackage{mathrsfs}
\usepackage{amssymb}
\usepackage{amsmath,amsfonts,amsthm}
\usepackage{graphicx}
\usepackage{upgreek}
\usepackage[toc,page]{appendix}
\usepackage{bbm}
\usepackage{esint}
\usepackage{scalerel}
\usepackage{stackengine,wasysym}
\usepackage[font=small,labelfont=bf]{caption} 
\usepackage{verbatim, latexsym, amssymb, amsmath,color}
\usepackage{epsfig}
\usepackage{hyperref}
\usepackage{float}
\usepackage{graphicx}
\usepackage{subfig}
\usepackage{mathtools}
\usepackage{tikz-cd}
\usepackage[normalem]{ulem}
\usepackage{marginnote}
\usepackage{setspace}
\usepackage{enumitem}

\newtheorem{thm}{Theorem}[section]
\newtheorem{lemm}[thm]{Lemma}
\newtheorem{cor}[thm]{Corollary}

\newtheorem{prop}[thm]{Proposition}
\newtheorem{lem}[thm]{Lemma}

\newtheorem{teo}{Theorem}[section]

\theoremstyle{remark}
\newtheorem{rmk}[thm]{Remark}

\theoremstyle{definition}
\newtheorem{definition}[thm]{Definition}
\newtheorem{exam}[thm]{Example}

\numberwithin{equation}{subsection}

\usepackage{hyperref}
\usepackage[dvipsnames]{xcolor}
\hypersetup{colorlinks=true, linktocpage=true,citecolor=Emerald, linkcolor=RedViolet, urlcolor=RedViolet}

\allowdisplaybreaks

\title[Ambient geometry via min-max widths of embedded circles]{Ambient geometry via min-max widths of embedded circles}

\author{Lucas Ambrozio, Ivan Miranda, and Rafael Montezuma}
\address{IMPA, Rio de Janeiro RJ 22460-320 Brazil}
\email{l.ambrozio@impa.br}
\email{ivan.miranda@impa.br}
\address{Departamento de Matemática, UFC, Fortaleza/CE 60455-760 Brazil}
\email{montezuma@mat.ufc.br}

\thanks{L.A. was supported by CNPq (302815/2025-2 - Bolsa PQ
and 406666/2023-7 - Universal) and by FAPERJ (E-26/200.175/2023 - Bolsa JCNE and E-26/204.290/2025 - Bolsa JCNE). I.M. was financed in part by the Coordenação de Aperfeiçoamento de Pessoal de Nível Superior - Brasil (CAPES) – Finance Code  001. R.M. is supported by CNPq (305368/2024-9 - Bolsa PQ
and 406666/2023-7 - Universal) and Instituto Serrapilheira grant “New perspectives of the min-max theory for the area functional”}

\begin{document}
	
	\begin{abstract}
		We prove a lower bound for the Birkhoff min-max invariant of a Riemannian sphere in terms of the min-max width of its embedded circles. The main tool is a method to induce a sweepout by pairs of points in an embedded circle from a given sweepout of the sphere by closed curves, so that the points of each pair of the induced sweepout lie close to two curves of the ambient sweepout that are close to each other. Moreover, considering a non-compact complete Riemannian manifold, we relate upper bounds for the min-max width of its embedded circles to the existence of continuous functions from the manifold to finite graphs with level sets that have uniformly bounded diameters, thus giving an estimate for its $1$-width in Urysohn's quantitative dimension theory. In the specific case of the Euclidean plane, we bound from below the classical width of a simple closed curve by its min-max width. Finally, we prove related rigidity results characterizing the round spheres among Zoll spheres.
	\end{abstract}
	
	\maketitle 
	
	\setcounter{tocdepth}{1}
	\tableofcontents  

    \section{Introduction}
    The classical width of a compact planar region is the shortest distance between two parallel lines that contain the region inside the slab between them. The width of a (smooth) simple closed planar curve is the width of the compact region that it encloses. When the curve is convex, its width is attained as the length of a line segment that touches the curve at both ends in right angles. Thus, in that case, this number is a critical value of the Euclidean distance between pairs of points of the curve. 

    This variational perspective was one motivation for the joint work of the first and third named authors with R. Santos \cite{AmbMonSan}. There, a Morse-Lusternik-Schinirelmann theory was developed for the distance functional defined on the space of pairs of points of an embedded circle in a complete Riemannian manifold. In particular, in this theory, there is a min-max notion of width, which is realized by a pair of points on the circle that is critical for the distance functional in a generalized sense, inspired by the work of Grove and Shiohama \cite{GroShi}.

    Let $\gamma$ be an embedded circle in a complete Riemannian manifold $(M,g)$. The space $\mathcal{P}_{\gamma}^*$ of subsets of $\gamma$ with one or two elements, where all singletons are identified, is homeomorphic to the real projective plane. A sweepout of $\gamma$ is a continuous family $\{p_t,q_t\}_{t\in [0,1]}$, consisting of pairs of points of $\gamma$ bounding closed arcs which vary continuously from the trivial arc $\{p_0\}=\{q_0\}$ to the entire circle. Every such family is a non-trivial loop when regarded as a path in $\mathcal{P}_{\gamma}^*$.
(See Definition \ref{defiweepout} and Proposition \ref{prop-sweepout-characterization} for details). The min-max width of $\gamma$ introduced in \cite{AmbMonSan} is the number
\begin{equation*}
    \mathcal{S}_g(\gamma) = \inf_{\substack{\{p_t,q_t\} \\ \text{sweepout}}} \left(\max_{t\in [0,1]} d_g(p_t,q_t)\right),
\end{equation*}
where $d_g(p_t,q_t)$ is the distance in $(M,g)$ between the points $p_t$, $q_t\in \gamma$. It was proven in \cite{AmbMonSan} that $\mathcal{S}_g(\gamma)$ is a non-trivial critical value of the distance functional on $\mathcal{P}^*_\gamma$, and that every critical pair at that level is connected by a simultaneously stationary collection of minimizing geodesics (see \cite{AmbMonSan}, Theorem A). This condition means that it is not possible to decrease to first order the lengths of all of the geodesics in the collection with respect to the flow of the same vector field. When the critical pair is connected by a unique minimizing geodesic, this collection consists of a free boundary geodesic with respect to $\gamma$.

        For convex closed planar curves, a class of curves that include curves of constant width, the min-max width is equal to the classical width (see Corollary 1.4 in \cite{AmbMonSan}). Moreover, an embedded circle in a Riemannian manifold whose min-max width coincides with its extrinsic diameter were regarded in \cite{AmbMonSan} as a Riemannian analogue of a planar curve of constant width, and classical theorems about the planar curves of constant width were proven in greater generality. (See, for instance, \cite{MarMonOli} for the classical theory, and Theorems C and F in \cite{AmbMonSan} for the results in the Riemannian setting).

    Our main goal is to investigate the interplay between the geometry of a complete Riemannian manifold and the min-max widths of the embedded circles it contains.  

    Our first result is about Riemannian spheres. A sweepout of $\mathbb{S}^2$ is a path $\{\Sigma_t\}_{t \in [-1,1]}$ of piecewise smooth closed curves of the form $\Sigma_t = F\circ \overline{c}_t$ for some continuous map $F : \mathbb{S}^2_0 \rightarrow \mathbb{S}^2$ of degree one, where $\{\overline{c}_t\}$ is the standard sweepout of the unit sphere $\mathbb{S}^2_0 \subset \mathbb{R}^3$ by horizontal circles $\overline{c}_t = \mathbb{S}^2_0 \cap \{z=t\}$. The Birkhoff invariant of the Riemannian sphere $(\mathbb{S}^2, g)$ is the following min-max quantity,
\begin{equation*}
w_{B}(\mathbb{S}^2, g) = \inf_{ \substack{ \{ \Sigma_t\} \\ \text{sweepout} }}  \left(\sup_{t\in [-1,1]} L_g(\Sigma_t)\right) , 
\end{equation*}
where $L_g(\Sigma_t)$ is the length of the closed curve $\Sigma_t$ with respect to $g$. It can be shown, following the pioneering arguments of Birkhoff, that $w_{B}(\mathbb{S}^2, g)$ is the length of a closed immersed geodesic \cite{Bir,Cro}. If $(\mathbb{S}^2,g)$ has positive Gaussian curvature, then $w_B(\mathbb{S}^2,g)$ is the length of the shortest non-trivial closed geodesic in $(\mathbb{S}^2,g)$, which is known to be embedded, by the work of Calabi and Cao \cite{Cal-Cao}. (See also the related work of Chambers and Liokumovich \cite{ChaLio}).

The task of estimating $w_B(\mathbb{S}^2,g)$ from above can be achieved by bounding from above the length of curves of explicit sweepouts. Thus, comparing the Birkhoff width to other geometric quantities is in general a problem of constructing sweepouts cleverly. (For instance, see Treibergs' comparison between $w_B(\mathbb{S}^2,g)$ and the area \cite{Tre}. Compare also with \cite{Rot}). Estimating $w_B(\mathbb{S}^2,g)$ from below usually requires, on the other hand, a more involved analysis, as it is typical for min-max invariants. (Compare with \cite{BalSab} or \cite{Che}). In our main result, we prove a lower bound for the Birkhoff invariant of a Riemannian sphere in terms of the min-max width of its embedded circles:

	{
		\renewcommand{\theteo}{A}
		\begin{teo} \label{main-theorem} 
        			Let $(\mathbb{S}^2, g)$ be a Riemannian sphere. For every embedded circle $\gamma$ in $\mathbb{S}^2$,
            \begin{equation*}
                w_{B}(\mathbb{S}^2, g)\geq 2 \cdot \mathcal{S}_g(\gamma).
            \end{equation*}
		\end{teo}                                  
	}
    
    The inequality in Theorem \ref{main-theorem} is sharp, in the sense that it is attained in certain Riemannian spheres by certain embedded circles. For instance, equality holds in Theorem \ref{main-theorem} for the round spheres and its great circles. It does not characterize this geometry, though: the equality also holds, more generally, for positively curved spheres that are invariant by the antipodal map (see Example \ref{example-antipodal-map}). 

    Given a continuous map $\varphi: \gamma\rightarrow \gamma$ without fixed points, Lemma 5.1 in \cite{AmbMonSan} shows that every sweepout $\{p_t, q_t\}_{t \in [0,1]}$ of $\gamma$ contains a pair of points of the form $\{x_{\ast}, \varphi(x_{\ast})\}$, for some $x_{\ast} \in \gamma$. In particular, $\mathcal{S}_g(\gamma)\geq \min_{x\in \gamma} d_g(x,\varphi(x))$. Therefore, Theorem \ref{main-theorem} can be an effective tool for estimating the Birkhoff invariant of a Riemannian sphere.

    Aside from the Birkhoff invariant, there is another min-max geometric invariant associated to the length functional on the space of curves in a Riemannian sphere, namely, its Almgren-Pitts min-max width (\cite{Pitts, Cal-Cao}). The main difference between the two invariants is the much larger class of sweepouts allowed by Almgren-Pitts theory, which includes, for instance, sweepouts by disjoint unions of closed curves. In contrast to Theorem \ref{main-theorem}, we show that the three-legged starfishes idealized by Almgren and constructed by Mantoulidis and Marx-Kuo in \cite{Almgren-Starfish-Mantoulidis-Kuo} give examples of metrics on $\mathbb{S}^2$ admitting embedded circles whose min-max width can be arbitrarily large in comparison to the Almgren-Pitts min-max width (see Example \ref{example-almgren-pitts}). These three-legged starfishes converge to a thrice punctured sphere with a complete Riemannian metric. The embedded circles that resemble geodesic triangles whose vertices lie arbitrarily far away near the three distinct ends have arbitrarily large min-max width.
    
    More generally, complete Riemannian manifolds whose embedded circles have uniformly bounded min-max width not only have their topology restricted, but also their geometry is restricted in such way that it resembles, macroscopically, a finite graph.

    	{
		\renewcommand{\theteo}{B}
        
    \begin{teo} \label{theorem-non-compact}
       Let $(M,g)$ be a complete non-compact Riemannian manifold. The following assertions are equivalent:
        \begin{enumerate}
            \item $\sup \{ \mathcal{S}_g(\gamma) : \gamma \text{ is an embedded circle in } M\} < +\infty$.
            \item The manifold $M$ has at most two ends, and there exist a finite graph $G$ and a continuous function $f: M \to G$ such that 
            \begin{equation*}
                \sup_{h \in G} diam f^{-1}(h) < +\infty.
            \end{equation*}
        \end{enumerate}
	\end{teo}     
    }

    Theorem \ref{theorem-non-compact} shows, in particular, that complete Riemannian manifolds, of arbitrary dimension, whose embedded circles have uniformly bounded min-max width are one-dimensional in the sense of Urysohn's topological dimensional theory, as it gives a quantitative estimate for its $1$-width. (See Section 2 of \cite{Guth} for an introduction to topological dimension theory in connection to geometry and Section 2 of \cite{Bal} for more about the Urysohn width. We show in Example \ref{example-simplicial-complex-restriction} that the restriction on the simplicial complex $G$ in Theorem \ref{theorem-non-compact} is necessary).
    
    The right circular cylinder in Euclidean three-space is an example of non-compact surface where the largest width of an embedded circle is precisely the diameter of its circular slices (see Theorem \ref{thm-diam-->width}).

    Back to embedded circles in the Euclidean plane, a simple argument bounds from below the classical width $w(\gamma)$ of a simple closed planar curve by the inradius $inrad(\Omega)$ of the region $\Omega$ that the curve $\gamma$ encloses. With the tools we developed to prove Theorem \ref{main-theorem}, we show that $\mathcal{S}_g(\gamma)$ gives a better estimate:

        	{
		\renewcommand{\theteo}{C}
        
	\begin{teo} \label{thm-w>S}
		If $\gamma$ is a planar simple closed curve bounding a region $\Omega$, then
        \begin{equation*}
            w(\gamma) \ge \mathcal{S}(\gamma)\geq 2\cdot inrad(\Omega).
        \end{equation*}
    \end{teo}
    }

    If $\gamma$ is the planar simple closed curve with the shape of a thin letter $U$, then $w(\gamma) > \mathcal{S}(\gamma)$. An equilateral triangle with slightly rounded corners gives an example where $\mathcal{S}(\gamma)>2\cdot inrad(\Omega)$.

    Our last two results are rigidity results for Zoll spheres. They are Riemannian spheres whose geodesics are all simple closed curves with the same length. This is a very flexible condition: aside from the round Euclidean sphere, there are infinite dimensional families of Riemannian spheres with this property (see \cite{Bes}, \cite{Gui}, and \cite{Zol}).
    
    We show that the round spheres, among Zoll spheres, can be characterized by the min-max width of its geodesics: 

        	{
		\renewcommand{\theteo}{D}
        
    \begin{teo} \label{theorem-zoll-rigidity-intro}
		Let $(\mathbb{S}^2,g)$ be a Zoll sphere. If there exists a geodesic $\sigma$ of $(\mathbb{S}^2,g)$ such that
            \begin{equation*}
                \mathcal{S}_g(\sigma)=\frac{1}{2} \cdot w_B(\mathbb{S}^2,g),
            \end{equation*}
        then $(\mathbb{S}^2,g)$ has constant Gaussian curvature. 
	\end{teo}
    }

    The Birkhoff invariant of a Zoll sphere $(\mathbb{S}^2,g)$ is the common length of its geodesics. In contrast to Theorem \ref{theorem-zoll-rigidity-intro}, other geometric invariants are unable to characterize the round spheres among Zoll spheres. For instance, the diameter of a Zoll sphere is clearly bounded from above by half of its Birkhoff invariant, and Zoll spheres of revolution satisfy $diam(\mathbb{S}^2,g)=\frac{1}{2}w_B(\mathbb{S}^2,g)$. In contrast, Balacheff, Croke and Katz \cite{BalCroKat} constructed examples of Zoll spheres for which $diam(\mathbb{S}^2,g) < \frac{1}{2} w_B(\mathbb{S}^2)$. Incidentally, all embedded circles $\gamma$, and not only geodesics, satisfy $\mathcal{S}_g(\gamma)\leq diam(\mathbb{S}^2,g)< \frac{1}{2}w_B(\mathbb{S}^2,g)$ in these examples. In a similar vein, every Zoll sphere satisfies $w_B(\mathbb{S}^2,g)^2=\pi\cdot area(\mathbb{S}^2,g)$, by a result of Weinstein \cite{Wei}. 

    There is another min-max invariant for which we can prove a rigidity theorem for Zoll metrics. This is the min-max distance width of a Riemannian sphere $(\mathbb{S}^2,g)$, introduced by the third-named author and I. Ribeiro \cite{MonRib} as a higher dimensional analogue of the min-max width of embedded circles. 

    Let $(\mathbb{S}^2,g)$ be a Riemannian sphere. The space $\mathcal{P}_{\mathbb{S}^2}$ of subsets of $\mathbb{S}^2$ with one or two elements is homeomorphic to $\mathbb{C}\mathbb{P}^2$. A sweepout of $\mathbb{S}^2$ by pairs of points is a family $\{\phi_1(x),\phi_2(x)\}_{x\in \mathbb{S}^2}$, consisting of a pair of continuous functions $\phi_1,\phi_2 : \mathbb{S}^2 \to \mathbb{S}^2$, such that the corresponding map from $\mathbb{S}^2$ to $\mathcal{P}_{\mathbb{S}^2}$ is homotopic to the standard sweepout $\Phi_0(x) = \{x,x_0\}$, where $x_0 \in \mathbb{S}^2$ is fixed. The min-max distance width of a Riemannian sphere $(\mathbb{S}^2,g)$ is the real number 
    \begin{equation*}
        W_d(\mathbb{S}^2,g) = \inf_{\substack{\{\phi_1(x),\phi_2(x)\}\\ \text{sweepout}}} \left(\max_{x \in \mathbb{S}^2} d_g(\phi_1(x),\phi_2(x))\right).
    \end{equation*}
    It was shown in \cite{MonRib} that $W_d(\mathbb{S}^2,g)$ is a non-trivial critical value of the distance functional defined on $\mathcal{P}_{\mathbb{S}^2}$, in a generalized sense. When a critical pair at the level $W_d(\mathbb{S}^2,g)$ is connected by exactly two minimizing geodesics, they form together an embedded closed geodesic with length $2W_d(\mathbb{S}^2,g)$. In general, other configurations of simultaneously stationary minimizing geodesics may occur (for instance, see Theorem 1.8 in \cite{MonRib}). The case where the sphere is embedded in a complete Riemannian manifold and the distance function between pairs of points is the extrinsic distance was also considered in \cite{MonRib}. For convex spheres embedded in $\mathbb{R}^3$, a class that contains all bodies of constant width, this min-max width coincides with the classical width of the region it encloses (Proposition 1.14 in \cite{MonRib}), and classical results about bodies of constant width were generalized to the Riemannian setting (see Theorem 1.15 and Theorem 1.17 in \cite{MonRib}). 
    
    It turns out that the intrinsic geometric invariant $W_d(\mathbb{S}^2,g)$ also distinguishes the round metric among Zoll spheres:

        	{
		\renewcommand{\theteo}{E}
        
\begin{teo} \label{teo-E--Wd-vs-wB}
    Let $(\mathbb{S}^2,g)$ be a Zoll sphere. Then  $$w_B(\mathbb{S}^2,g) \geq 2 \cdot W_d(\mathbb{S}^2,g)$$ and equality holds if and only if $(\mathbb{S}^2,g)$ has constant Gaussian curvature.
\end{teo}
    }

    In fact, the inequality established in Theorem \ref{teo-E--Wd-vs-wB} holds more generally for any Riemannian sphere, as shown in Proposition \ref{thm-comparison} (this fact is essentially contained in \cite{MonRib}, Theorem 1.12). It is interesting to compare it with the inequality in Theorem \ref{main-theorem}. In Example \ref{example-supS>Wd} we construct Riemannian spheres $(\mathbb{S}^2,g)$ for which the estimate from below for $w_B(\mathbb{S}^2,g)$ in terms of the min-max width of certain embedded circles is strictly better than the estimate in terms of $W_d(\mathbb{S}^2,g)$.
    
    \begin{figure}[h]        \includegraphics[scale=.3]{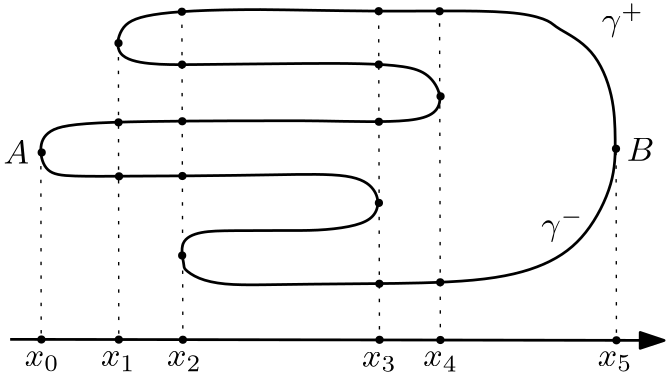}
            \caption{The planar closed curve $\gamma(t)=(x(t),y(t))$ above is in generic position with respect to the foliation by vertical lines, and each vertical line contains at most one critical point of the function $x(t)$. In order to produce a sweepout of $\gamma$ by pairs of points, starting at $A$ and ending at $B$, in such way that each pair is contained in the same vertical line, move one point in the arc $\gamma^+$ and the other point in the arc $\gamma^{-}$, so that their common coordinate $x(t)$ moves continuously forming the sequence $x_0$,  $x_3$, $x_2$, $x_4$, $x_2$, $x_3$, $x_1$, $x_3$, $x_2$, $x_5$, being monotone while moving between each two consecutive values. Theorem \ref{prop.algoritmo.v2} generalizes the basic algorithm behind this construction (see also Remark \ref{rmk-mountain-climbing}).}\label{figura}
    \end{figure}

    The proofs of all of the main theorems are underpinned by a method to induce a sweepout by pairs of points on an embedded circle under the presence of a Morse function such that each of its critical levels contains exactly one critical point (Theorem \ref{prop.algoritmo.v2}). Each pair of points of the induced sweepout is contained in the same level set of the Morse function. The proof of Theorem \ref{prop.algoritmo.v2} shows that there exists essentially a unique sweepout with this property. See Figure \ref{figura} for an illustration of it in the context of Theorem \ref{thm-w>S}. 
    
    In typical applications of this method, this Morse function is close to another function that encapsulates geometric information about the ambient manifold, and whose level sets have uniformly bounded diameters or lengths. For the proof of Theorem \ref{main-theorem}, on the other hand, we had to overcome new difficulties, arising from the fact that a general sweepout of a sphere does not define a foliation (see Theorem \ref{thm.compare-Birkhoff}).
    
    The rigidity results for Zoll spheres rely on a theorem of Bangert, which classifies compact Riemannian manifolds with boundary whose chords have all the same length \cite{Ban-chords}.

    We did not make use of LLMs to conduct the research that led to this article. This paper does not contain AI-generated text.

    The paper is organized as follows. Section \ref{section-Sweepouts by pairs of points} is a preliminary section, where we recall important concepts about sweepouts by pairs of points and discuss elementary facts about them. Section \ref{section::induced-sweepouts} is dedicated to the construction of sweepouts by pairs of points in the presence of Morse functions and to the first main application of this tool, namely, Theorem \ref{thm-w>S}. Section \ref{section:Sweepouts by pairs of points induced by Birkhoff sweepouts} explains how to construct a sweepout by pairs of points on an embedded circle out of a Birkhoff sweepout of a sphere, and how to use that construction to prove Theorem \ref{main-theorem}. Section \ref{section::non-compact-case} is dedicated to the analysis of non-compact manifolds. Section \ref{subsec-statements} contains the key results underlying the proof of Theorem \ref{theorem-non-compact}, which are of independent interest (see Theorem \ref{thm-widts-->dist-lines} and Theorem \ref{thm-diam-->width}). Section \ref{subsec-proof-thm-B} contains the proof of Theorem \ref{theorem-non-compact}, and Section \ref{subsec-further-res} contains another application of Theorem \ref{thm-diam-->width}, which concerns the min-max width of systoles and estimates for their lengths.
   Section \ref{section::rigidity-results} contains the rigidity results for Zoll spheres. In Section \ref{section::Examples}, we collect three examples, which are relevant to clarify the sharpness of our results.

    \section{Definitions and preliminary results}\label{section-Sweepouts by pairs of points}

 We recapitulate the definition and basic properties of sweepouts of embedded circles, and the definition and basic methods of estimation of the min-max width of embedded circles. This section is based in \cite{AmbMonSan} and \cite{AmbMon}.

Let $\gamma$ be an embedded circle in an  Riemannian manifold $(M,g)$. The space $\mathcal{P}_{\gamma}$ of subsets of $\gamma$ consisting of one or two points is homeomorphic to the M\"obius band, and the singletons are its boundary points. The real projective plane obtained by the quotient of $\mathcal{P}_{\gamma}$ by the identification of all singletons as a single point $[\ast]$ is denoted by $\mathcal{P}_{\gamma}^{\ast}$.

There is a simple geometric construction of the orientable $2$-covers of the spaces $\mathcal{P}_{\gamma}$ and $\mathcal{P}^*_{\gamma}$. It is useful in order to describe a meaningful interpretation of the class of sweepouts that is defined below. Fix an orientation on $\gamma$. An oriented closed arc $C\subset \gamma$ is a closed connected subset of $\gamma$ equipped with a choice of initial and final points. An oriented closed arc admits a unique positive parametrization $C : [0,1] \rightarrow \gamma$ proportional to arc length with possible self-intersections at the endpoints only, in the case that the initial and final points, $C(0)$ and $C(1)$ respectively, coincide. We use these parametrizations to endow the space of closed oriented arcs of $\gamma$ with the uniform convergence topology.

Let $\mathcal{O}$ denote the space of ordered pairs $(C_1,C_2)$ of oriented closed arcs of $\gamma$ such that: (a) the initial point of $C_2$ coincides with the final point of $C_1$, (b) the final point of $C_2$ coincides with the initial point of $C_1$, and (c) the concatenation of the parametrizations of $C_1$ and $C_2$ is a full parametrization of $\gamma$. Condition (c) implies that if one of the arcs $C_i$ degenerates to a point $x\in \gamma$, the other element in the pair is the arc of $\gamma$ starting and ending at $x$ that makes a complete positive turn around $\gamma$. The space $\mathcal{O}$ is homeomorphic to a cylinder.

Let $\mathcal{O}_*$ be the two-sphere obtained by the quotient of $\mathcal{O}$ by the equivalence relation that separately identifies all pairs $(C_1,C_2)\in \mathcal{O}$ where $C_1$ is a degenerate arc, and all pairs $(C_1,C_2)\in \mathcal{O}$ where $C_2$ is degenerate. In the quotient $\mathcal{O}_*$, the two points obtained by these identifications are denoted, respectively, by $([*],\gamma)$ and $(\gamma,[*])$. The topology of the space of oriented closed arcs induces topologies on $\mathcal{O}$ and, through the quotient map, on $\mathcal{O}_*$.
		
The map $\pi:\mathcal{O}\rightarrow \mathcal{P}_{\gamma}$, defined so that $\pi((C_1,C_2))$ is the set of the endpoints of either $C_1$ or $C_2$, is the oriented double cover of $\mathcal{P}_{\gamma}$. This map descends to the oriented double cover $\pi_*:\mathcal{O}_*\rightarrow \mathcal{P}_{\gamma}^*$.

\begin{definition}\label{defiweepout}

A family $\{p_t,q_t\}_{t\in [0,1]}$ of pairs of points of $\gamma$ is a \textit{sweepout} of $\gamma$ if the following properties are satisfied:
\begin{itemize}

\item[$(i)$] $p_0=q_0$ and $p_1=q_1$,

\item[$(ii)$] $t\in [0,1] \mapsto p_t\in \gamma$ and $t\in [0,1] \mapsto q_t\in \gamma$ are continuous mappings,

\item[$(iii)$] there exists a continuous family of oriented closed arcs  $C_t\subset \gamma$, for $t \in [0,1]$, whose initial and final points are $p_t$ and $q_t$, respectively. Moreover, $C_0$ is the degenerate arc represented by the point $p_0=q_0$, and $C_1$ is represented by the arc that covers $\gamma$ exactly once, with initial and final points $p_1=q_1$.
\end{itemize}
\end{definition}

Given a sweepout $\{p_t,q_t\}_{t \in [0.1]}$ of an embedded circle $\gamma$ in a Riemannian manifold $(M,g)$, it is clear that $t\in [0,1]\mapsto \{p_t,q_t\} \in \mathcal{P}_{\gamma}$ defines a path whose initial and final points lie in the boundary of $\mathcal{P}_\gamma$. Thus the quotient map $q^{\mathcal{P}}: \mathcal{P}_\gamma\rightarrow \mathcal{P}_\gamma^*$ maps it to a loop based at $[*]$. Using the spaces $\mathcal{P}_\gamma$, $\mathcal{P}_\gamma^*$ and their covers described earlier, it is straightforward to obtain the following criterion to detect paths in $\mathcal{P}_\gamma^*$ that are sweepouts of $\gamma$.

\begin{prop}\label{prop-sweepout-characterization}
Let $\gamma$ be an embedded circle in a Riemannian manifold $(M,g)$. If $\{p_t,q_t\}_{t \in [0,1]}$ is a sweepout of $\gamma$, then $t \in [0,1] \mapsto q^{\mathcal{P}}(\{p_t,q_t\})$ is a homotopically non-trivial loop in $\mathcal{P}_{\gamma}^*$ based at $[*]$.

Conversely, given a homotopically non-trivial loop $\psi:[0,1]\rightarrow \mathcal{P}_{\gamma}^*$ based at $[*]$, if there exists a continuous map $\phi : [0,1]\rightarrow \mathcal{P}_{\gamma}$ such that $q^{\mathcal{P}}\circ \phi = \psi$,
then there exists a sweepout $\{p_t,q_t\}_{t \in [0,1]}$ of $\gamma$ such that $\phi(t)=\{p_t,q_t\}$ for all $t \in [0,1]$.
\end{prop}

We recall next the definition of the min-max width of embedded circles introduced in \cite{AmbMonSan}. Let us use $d_g$ to represent the Riemannian distance in $(M,g)$. 

\begin{definition}\label{width-of-circle}
Let $\gamma$ be an embedded circle in a complete Riemannian manifold $(M,g)$. The min-max width of $\gamma$ is defined by
$$
\mathcal{S}_g(\gamma) = \inf_{\substack{\{p_t,q_t\} \\ \text{sweepout}}}\left(  \max_{t\in [0,1]} d_g(p_t,q_t) \right).
$$
\end{definition}

The next result, proven in \cite{AmbMonSan}, contains a basic and effective method to estimate the min-max width $\mathcal{S}_g(\gamma)$ from below.

\begin{lem}[See \cite{AmbMonSan}, Lemma 5.1]\label{lem-nofixedpts-AMS}
    Let $\varphi: \mathbb{S}^1\rightarrow \mathbb{S}^1$ be a continuous map without fixed points. Then, every sweepout $\{p_t, q_t\}$ of $\mathbb{S}^1$ contains a pair of points of the form $\{p_{t_0}, \varphi(p_{t_0})\}$, for some $t_0 \in [0,1]$.
\end{lem}

In order to illustrate the use of Lemma \ref{lem-nofixedpts-AMS}, we formulate the next proposition, which will be applied in the analysis of examples and in the proof of Theorem \ref{theorem-non-compact}.

\begin{prop}\label{prop-widthestimate-triangle}
Let $(M,g)$ be a complete Riemanninan manifold and $\gamma$ be a piecewise smooth embedded circle in $M$ with three distinct vertices $x$, $y$, and $z$. Let $\gamma_{xy}$, $\gamma_{yz}$ and $\gamma_{zx}$ denote the smooth closed arcs of $\gamma$ that join $x$ to $y$, $y$ to $z$ and $x$ to $z$, respectively. If there exists $L>0$ such that 
\begin{equation*}
    d(x, \gamma_{yz}), d(y, \gamma_{zx}), d(z, \gamma_{xy}) > L,
\end{equation*}
    then there exists an embedded circle $\tilde{\gamma}$, arbitrarily close to $\gamma$ in the Hausdorff topology, such that $S(\tilde{\gamma}) > L$.  
\end{prop}

\begin{proof}
    By slightly rounding off the curve $\gamma$ near the points $x$, $y$, and $z$, one obtains an embedded circle $\tilde{\gamma}$ for which $S(\tilde{\gamma})>L$. Indeed, it is straightforward to construct a continuous involution $\varphi$ of $\tilde{\gamma}$ that maps points near the vertex $x$ to points that lie on the smooth closed arc $\gamma_{yz}$, and analogously for points near $y$ or $z$. In particular, $d(p, \varphi(p)) > L$ for all $p\in \tilde{\gamma}$. Lemma \ref{lem-nofixedpts-AMS} implies that every sweepout of $\tilde{\gamma}$ contains a pair of points of the form $\{p_0, \varphi(p_0)\}$ for some $p_0 \in \tilde \gamma$. Thus $\mathcal{S}_g(\tilde \gamma) > L$.

\end{proof}

\section{Sweepouts by pairs of points induced by Morse functions} \label{section::induced-sweepouts}

    The main goal of this section is to develop a method to construct a sweepout of an embedded circle by pairs of points given a Morse function. This tool, Theorem \ref{prop.algoritmo.v2}, will be used in all subsequent sections. We give a first immediate application of it, namely, Theorem \ref{thm-w>S}.

    \begin{thm}\label{prop.algoritmo.v2}
    Let $f:\mathbb{S}^1\rightarrow \mathbb{R}$ be a Morse function such that each critical level contains a unique critical point. Then there exists a sweepout $\{p_t, q_t\}_{t\in [0,1]}$ of $\mathbb{S}^1$ such that $f(p_t) = f(q_t)$, for every $t \in [0,1]$.
    
    \end{thm}

    \begin{rmk} \label{rmk-mountain-climbing}
        The above theorem and its proof are closely related to the so-called Mountain Climbing Problem. We only became aware of it after the first version of this paper was made available in the arXiv. We thank Caio B. Rodrigues for drawing our attention to it. 
    \end{rmk}
    \begin{proof}
        Let $\mathcal{P} = \mathcal{P}_{\mathbb{S}^1}$ and $\mathcal{O}$ be as introduced in  Section \ref{section-Sweepouts by pairs of points}. Let $int(\mathcal{O})$ denote the interior of the cylinder and consider the smooth function $F : int(\mathcal{O}) \rightarrow \mathbb{R}$ defined by $F((C_1, C_2)) := f(p_1) - f(p_0)$ where $p_0$ and $p_1$ are the initial and final points of the arc $C_1$, assuming that this arc is parametrized in the counterclockwise direction. 

We claim that $0$ is a regular value of $F$. Indeed, given $(C, D) \in F^{-1}(0)$, consider a monotone smooth parametrization $\alpha : [-\varepsilon, 1+\varepsilon] \rightarrow \mathbb{S}^1$ in the counterclockwise direction such that $\alpha(0) = p_0 \neq \alpha(1) = p_1$ are the initial and final points of $C$, respectively, and whose image is a proper arc of $\mathbb{S}^1$. Let $(C_t, D_t) \in \mathcal{O}$ be such that $C_t$ is the arc beginning at $\alpha(t)$ and ending at $\alpha(t+1)$, for $-\varepsilon \leq t \leq \varepsilon$. Possibly after decreasing $\varepsilon$, if necessary, we have $(C_t, D_t) \in int(\mathcal{O})$ and 
$$
\left.\frac{d}{dt}\right|_{t=0} F((C_t, D_t)) = df_{p_1} (\alpha'(1)) - df_{p_0} (\alpha'(0)). 
$$
Since $F((C_0, D_0)) = F((C, D)) = 0$, we have $f(p_1)=f(p_0)$. In particular, by hypothesis, at most one of those different points is a critical point of $f$. Therefore, it is possible to choose $\alpha$ in such a way that the derivative above is different from zero. Since $(C,D)\in F^{-1}(0)$ is arbitrary, the claim is proved.

Consider the set
    \begin{equation*}
        \Omega_f := \Big\{\{p,q\} \in \mathcal{P}\setminus \partial \mathcal{P} : f(p)=f(q) \Big \} \cup \Big \{\{p\} \in \mathcal{P}: p \text{ is critical point of } f \Big\}.
    \end{equation*}

The set $F^{-1}(0)$ is invariant under the deck transformation $(C_1,C_2) \mapsto (C_2,C_1)$. Thus $\Omega_f\setminus \partial \mathcal{P}=\pi(F^{-1}(0))$ is a smooth one-dimensional embedded submanifold, where $\pi:\mathcal{O}\rightarrow \mathcal{P}$ is the projection. We show that $\Omega_f$ is the closure of the subset $\pi(F^{-1}(0))\subset \mathcal{P}$. 

On the one hand, if $p$ is not a critical point of $f$, then $\{p\}$ is not in the closure of $\pi(F^{-1}(0))$ because $f$ is a diffeomorphism in a neighborhood of $p$. On the other hand, whenever $p$ is a critical point of $f$, the fact that $f$ is Morse gives us a parametrization $\beta$ of $\mathbb{S}^1$ near $p = \beta(0)$ such that $f(\beta(t)) = f(p) \pm t^2$. Then, there exists a neighborhood $U$ of $\{p\} \in \mathcal{P}$ and $\delta>0$ such that $\pi(F^{-1}(0))\cap U$ is the image of the path $t \in (0, \delta) \mapsto \{\beta(t), \beta(-t)\}$.

It follows from the above argument that $\Omega_f$ is a disjoint union of circles in $\mathcal{P}\setminus\partial \mathcal{P}$ and open intervals in $\mathcal{P}\setminus\partial \mathcal{P}$ whose extremities are connecting points $\{p\}\in\partial \mathcal{P}$ with $p$ a critical point of $f$. Let $\Omega_{\ast}$ denote the connected component of $\Omega_f$ that contains the point of $\partial \mathcal{P}$ corresponding to the point $A$ of global minimum of $f$. By the structure of $\Omega_f$, $\Omega_{\ast}$ has a second distinct point of accumulation on $\partial\mathcal{P}$, which corresponds to another critical point of $f$. We claim that it must be the point $B$ associated to the highest value of $f$. Indeed, the interior points of $\Omega_{\ast}$ are made of pairs $\{p,q\}$ with $p\neq q$ in different components of $\mathbb{S}^1\setminus \{A,B\}$. This holds for points of $\Omega_* \setminus \partial \mathcal{P}$ close to $\{A\}$ by the argument of the previous paragraph, and therefore holds in the connected set $\Omega_* \setminus \partial \mathcal{P}$. The other extremity of $\Omega_*$ is then either $A$ or $B$, but the argument in the previous paragraph excludes the first case.

Let $\phi: [0,1] \to \Omega_*$ be a homeomorphism with $\phi(0)=\{A\}$ and $\phi(1)=\{B\}$. In order to see that the resulting loop $\psi=q^{\mathcal{P}}\circ \phi:[0,1]\rightarrow \mathcal{P}^*$ determines sweepout, it is enough to check that $\psi$ is a non-trivial loop, by Proposition \ref{prop-sweepout-characterization}. Let $t\in[0,1]\mapsto (C_t,D_t)\in \mathcal{O}$ be the lift of $\phi$ where the initial arc $C_0$ is the degenerate arc with both endpoints at $A$. Notice that all pairs $\phi(t)$ have points in each of the two arcs of $\mathbb{S}^1$ determined by $A$ and $B$. This implies that $A$ belongs to the arcs $C_t$, for all $t\in[0,1]$. In particular, the final pair of arcs $(C_1, D_1)$ is such that $C_1$ starts and ends at $B$ and contains $A$. Therefore, $(C_1, D_1)$ is in the equivalence class $(\mathbb{S}^1, [\ast])$, while $(C_0,D_0)$ is in the equivalence class $([\ast],\mathbb{S}^1)$. Proposition \ref{prop-sweepout-characterization} implies that $\psi$ defines a sweepout $\{p_t,q_t\}_{t\in [0,1]}$ of $\mathbb{S}^1$.
\end{proof}

    Theorem \ref{prop.algoritmo.v2} is the key tool for the proof of Theorem \ref{thm-w>S}. Let $\gamma:\mathbb{S}^1\rightarrow \mathbb{R}^2$ be a simple closed curve. Given a unit vector $v \in \mathbb{S}^1$  the width of $\gamma$ with respect to $\theta$, denoted by $w_{v}(\gamma)$, is the minimum of $r>0$ such that there exist two lines perpendicular to $v$ which are at distance $r$ and contain the trace of $\gamma$ in the region between them. Recall that $w_v(\gamma)$ is a continuous function of $v$, and that the (classical) width of $\gamma$ is 
    \begin{equation*}
        w(\gamma) = \min \{w_v(\gamma) : v\in \mathbb{S}^1\}. 
    \end{equation*}

    \begin{thm}\label{thm.algoritmo}
    Let $\gamma$ be a simple closed planar curve. Then $\mathcal{S}(\gamma)\leq w (\gamma)$. 
    \end{thm}
    
 \begin{proof}
    Given $v\in \mathbb{S}^1$, let $f_v : \mathbb{R}^2\rightarrow \mathbb{R}$ be the map defined by $f_v(p) = \langle p, v\rangle$. Let $X$ be the dense set of vectors $v\in \mathbb{S}^1$ such that $f_v\circ \gamma:\mathbb{S}^1\rightarrow \mathbb{R}$ is a Morse function such that each critical level contains a single critical point. Given $v\in X$, let $u\in \mathbb{S}^1$ be orthogonal to $v$. Theorem \ref{prop.algoritmo.v2} applied to $f_v\circ \gamma$ gives a sweepout $\{p_t,q_t\}_{t\in [0,1]}$ of $\mathbb{S}^1$ such that $\langle v,\gamma(p_t)-\gamma(q_t)\rangle=0$ for every $t\in [0,1]$. It follows that $\{\gamma(p_t),\gamma(q_t)\}_{t\in [0,1]}$ is a sweepout of $\gamma$ and $|\gamma(p_t)-\gamma(q_t)|\leq w_u(\gamma)$ for every $t\in [0,1]$. The result now follows by the definition of $\mathcal{S}(\gamma)$ and $w(\gamma)$, using the denseness of $X$.
 \end{proof}

    Recall that the inradius of a planar domain is the supremum of the radii of the disks it contains.

    \begin{thm}\label{thm:inradius}
        Let $\Omega$ be a planar domain bounded by a simple closed curve $\gamma$. Then $2\cdot inrad(\Omega)\leq \mathcal{S}(\gamma)$.
    \end{thm}
    \begin{proof}
        Let $D\subset \Omega$ be an open disc. Denote by $\pi:\mathbb{R}^2\setminus D \rightarrow \partial D$ the radial projection. Given a sweepout $\{p_t,q_t\}_{t\in [0,1]}$ of $\gamma$, $\{\pi(p_t),\pi(q_t)\}_{[0,1]}$ is a sweepout of $\partial D$, because the restriction $\pi:\gamma\rightarrow \partial D$ is a degree one map. Moreover, $|\pi(p_t)-\pi(q_t)|\leq |p_t-q_t|$, by the convexity of $D$. It follows that the min-max width of $\partial D$, which is simply its diameter, is at most equal to the min-max width of $\gamma$. Since $D$ is arbitrary, the result follows.
    \end{proof}
    Combining Theorem \ref{thm.algoritmo} and Theorem \ref{thm:inradius}, we have proven Theorem \ref{thm-w>S}.

\section{Sweepouts by pairs of points induced by Birkhoff sweepouts}  \label{section:Sweepouts by pairs of points induced by Birkhoff sweepouts}

The first theorem of this section is the key ingredient of the proof of Theorem \ref{main-theorem}. The main new difficulty is that sweepouts of the sphere, in general, do not define foliations. And yet, Theorem \ref{thm.compare-Birkhoff} shows that it is still possible to construct a sweepout on an embedded circle in a Riemannian sphere, whose points are close to leaves of a given sweepout of the sphere, which are themselves close to each other.

We use $d_g$ and $d_H$ to denote the Riemannian distance and the Hausdorff distance in $(\mathbb{S}^2,g)$.

\begin{thm}\label{thm.compare-Birkhoff}
Let $\gamma$ be an embedded circle in a Riemannian sphere $(\mathbb{S}^2, g)$. Let $\{\Sigma_s\}_{s\in [-1,1]}$ be a sweepout of $\mathbb{S}^2$. Given any $\varepsilon>0$, there exists a sweepout $\{x_t, y_t\}_{t\in [0,1]}$ of $\gamma$ with the following property: there are continuous  functions $s_x,s_y:[0,1] \to [-1,1]$ such that, for every $t\in [0,1]$,
\begin{equation*}
    d_g(x_t, \Sigma_{s_x(t)}), d_g(y_t, \Sigma_{s_y(t)}),d_H(\Sigma_{s_x(t)},\Sigma_{s_y(t)})<\varepsilon.
\end{equation*}

\end{thm}

    \begin{proof} 
    Throughout this proof, we use $\mathbb{S}^2_0 = \{(x,y,z) \in \mathbb{R}^3 : x^2+y^2+z^2=1\}$.  Let $F: \mathbb{S}^2_0 \rightarrow \mathbb{S}^2$ be a continuous degree one map such that $\Sigma_s = F\circ \overline{c}_s$, for every $s \in [-1,1]$. Choose $0< \tau < 1/2$ so that
\begin{equation}\label{eq-thm.compare-Birkhoff-1}
\text{if } \theta_1, \theta_2\in \mathbb{S}^1 \text{ and } d_{\mathbb{S}^1}(\theta_1, \theta_2)< \tau \text{, then } d_{g} (\gamma(\theta_1),\gamma(\theta_2)) < \varepsilon/2; \text{ and}
\end{equation}
\begin{equation}\label{eq-thm.compare-Birkhoff-2}
\text{if } s_1, s_2 \in [-1,1] \text{ and } |s_1 - s_2| < \tau \text{, then } d_H (\Sigma_{s_1}, \Sigma_{s_2}) < \varepsilon/2.
\end{equation} 

Letting $z(p)$ represent the $z$-coordinate of $p \in \mathbb{S}^2_0$, consider the auxiliary maps
$$
\hat{F} : \mathbb{S}^2_0 \rightarrow \mathbb{S}^2 \times \mathbb{R}, \quad \hat{F}(p) := (F(p), z(p)),
$$
and
$$
\hat{\gamma} : \mathbb{S}^1 \times \mathbb{R} \rightarrow \mathbb{S}^2 \times \mathbb{R}, \quad \hat{\gamma}(\theta, s):= (\gamma(\theta), s).
$$
Let $\text{Im}(\hat{F})$ be the image of the map $\hat{F}$, and consider the compact set $\Omega$ of the points $(\theta, s)$ with $\gamma(\theta) \in \Sigma_s$. Equivalently, 
\begin{equation*}
\Omega = \{(\theta, s) \in \mathbb{S}^1\times [-1,1] : \hat{\gamma}(\theta,s) \in \text{Im}(\hat{F})\}.
\end{equation*}
Let $d$ denote the standard distance function on $\mathbb{S}^1\times \mathbb{R}$. \\

\noindent \textbf{Claim:} 
There exists a homotopically non-trivial smoothly embedded curve $\sigma : \mathbb{S}^1 \rightarrow \mathbb{S}^1\times [-1,1]$ satisfying $$d(\sigma(u), \Omega)<\tau \quad \text{for all} \quad u \in \mathbb{S}^1.$$ 

\begin{proof}[Proof of the Claim:]
Let $p_{\pm} = (\theta_{\ast}, \pm 2)$, for a fixed $\theta_{\ast} \in \mathbb{S}^1$. Let $U$ be the connected component of $\left(\mathbb{S}^1\times \mathbb{R}\right)\setminus \Omega$ which contains $p_+$. Let us prove that $p_- \notin U$. Suppose, by contradiction, that there exists an embedded path $\alpha : [-2,2] \rightarrow U$ such that $\alpha(-2)=p_-$ and $\alpha(2)=p_+$. Let $\overline{\alpha} : \mathbb{R} \rightarrow U$ be defined by $\overline{\alpha}(t) = \alpha(t)$, for $t\in [-2,2]$, and $\overline{\alpha}(t) = (\theta_{\ast}, t)$, for $|t|>2$. We can assume, without loss of generality, that $\alpha([-2,2]) \subset \mathbb{S}^1\times [-2,2]$, which implies that $\overline{\alpha}$ is embedded. In particular, $\hat{\gamma} \circ \overline{\alpha}$ is also an embedded in $\mathbb{S}^2\times \mathbb{R}$. Clearly, $\left(\mathbb{S}^2\times \mathbb{R}\right)\setminus \text{Im}(\hat{\gamma} \circ \overline{\alpha})$ is contractible.  On the other hand, $\alpha \subset U$ implies that $\text{Im}(\hat{F})$ does not intersect $\text{Im}(\hat{\gamma} \circ \overline{\alpha})$. We conclude that $\hat{F}$ is homotopic to a constant map. But this contradicts the fact that $F$ has degree one. Thus, $p_- \notin U$.

Consider the signed distance function to $\partial U$ (which is contained in $\mathbb{S}^1\times [-1,1]$) on the cylinder $\mathbb{S}^1\times \mathbb{R}$, constructed as follows. Let $\eta : \mathbb{S}^1 \times \mathbb{R} \rightarrow \mathbb{R}$ be defined by $\eta(x) = d(x, \partial U)$, whenever $x\in \overline{U}$, and $\eta(x) = - d(x, \partial U)$, otherwise. Given $0< \delta< \tau/2$, there exists a smooth function $\tilde{\eta} : \mathbb{S}^1 \times \mathbb{R} \rightarrow \mathbb{R}$ such that $|\eta - \tilde{\eta}|_{C^0} < \delta$, for which $0$ is a regular value. Notice that $\eta(p_-) \leq -1$ and $1\leq \eta (p_+)$, so that $\tilde{\eta}(p_-) < -1/2$ and $1/2 < \tilde{\eta} (p_+)$. Moreover, $\tilde{\eta}^{-1}(0)$ is compact, and therefore a union of embedded circles in $\mathbb{S}^1\times \mathbb{R}$. The inequalities $\tilde{\eta}(p_-) < 0 < \tilde{\eta}(p_+)$ imply that $\tilde{\eta}^{-1}(0)$ contains a homotopically non-trivial circle $\sigma : \mathbb{S}^1 \rightarrow \mathbb{S}^1\times \mathbb{R}$. Otherwise, the Jordan Theorem would imply that each circle of $\tilde{\eta}^{-1}(0)$ would separate $\mathbb{S}^1\times \mathbb{R}$ into two components, one bounded and one unbounded. Therefore, $\tilde{\eta}$ would have the same sign on both ends. This contradicts the preceding inequalities.

Therefore, for every $u \in \mathbb{S}^1$, we have
$$
d(\sigma(u), \Omega)\leq d(\sigma(u), \partial U)=|\eta(\sigma(u))| = |\eta(\sigma(u)) - \tilde{\eta}(\sigma(u))| < \delta.
$$

If $\partial U \subset \mathbb{S}^1\times (-1,1)$, then, for sufficiently small $\delta$, the embedded circle $\sigma$ is also contained in $\mathbb{S}^1\times (-1,1)$, and the claim is proven. Otherwise, $\sigma$ is contained in $\mathbb{S}^1\times (-1-\delta,1+\delta)$. Let us write $\sigma(u) = (\theta(u), s(u))$, where $\theta(u) \in \mathbb{S}^1$ and $s(u) \in \mathbb{R}$. Consider the embedded curve $\tilde{\sigma}: \mathbb{S}^1\rightarrow \mathbb{S}^1\times [-1,1]$ defined by $\tilde{\sigma}(u) = (\theta(u), (1+\delta)^{-1}s(u))$. It clearly satisfies $d(\tilde{\sigma}(u), \sigma(u))=|s(u)|\delta/(1+\delta)$. In particular, for sufficiently small $\delta$, we have $d(\tilde{\sigma}(u), \sigma(u))<\tau/2$, for every $u\in \mathbb{S}^1$. Therefore, $d(\tilde{\sigma}(u), \Omega)<\tau$ for all $u\in \mathbb{S}^1$, and $\tilde{\sigma}$ satisfies the desired properties.
\end{proof}

  We write $\sigma(u) = (\theta(u), s(u))$, where $\theta(u) \in \mathbb{S}^1$ and $s(u) \in \mathbb{R}$ and let $f: \mathbb{S}^1\rightarrow \mathbb{R}$ be a Morse function such that each critical level of it contains a unique critical point, and such that $|f(u) - s(u)| < \tau/2$, for all $u$. Theorem \ref{prop.algoritmo.v2} provides us with a sweepout $\{p_t, q_t\}_{t \in [0,1]}$ of $\mathbb{S}^1$ such that $f(p_t) = f(q_t)$, for all $t \in [0,1]$. In particular, $|s(p_t) - s(q_t)| < \tau$ and, by property (\ref{eq-thm.compare-Birkhoff-2}), 
\begin{equation*}
d_H(\Sigma_{s(p_t)}, \Sigma_{s(q_t)}) < \varepsilon/2.
\end{equation*}
By the Claim, $d(\sigma(u), \Omega)<\tau$ for every $u\in \mathbb{S}^1$. In particular, there exist $(\theta_t^p, s_t^p) \in \Omega$, $(\theta_t^q, s_t^q) \in \Omega$ (and hence $\gamma(\theta^p_t) \in \Sigma_{s_t^p}$ and $\gamma(\theta^q_t) \in \Sigma_{s_t^q}$), such that, for all $t \in [0,1]$,
\begin{equation*}
d((\theta(p_t), s(p_t)),(\theta_t^p, s_t^p)) < \tau \text{ and } d( (\theta(q_t), s(q_t)),(\theta_t^q, s_t^q)) < \tau.
\end{equation*}
As a consequence of the properties of the distance in the cylinder, we have
$$
d_{\mathbb{S}^1}(\theta_t^p, \theta(p_t)) < \tau \text{ and } |s_t^p - s(p_t)| < \tau.
$$
Therefore, in addition to $\gamma(\theta^p_t) \in \Sigma_{s_t^p}$ and $\gamma(\theta^q_t) \in \Sigma_{s_t^q}$, equations (\ref{eq-thm.compare-Birkhoff-1}) and (\ref{eq-thm.compare-Birkhoff-2}) yield
\begin{equation*}
d_{g}(\gamma(\theta_t^p), \gamma(\theta(p_t)))<\varepsilon/2 \text{ and }  d_H (\Sigma_{s_t^p}, \Sigma_{s(p_t)}) < \varepsilon/2.
\end{equation*}
Similarly, 
\begin{equation*}
d_{g}(\gamma(\theta_t^q), \gamma(\theta(q_t)))<\varepsilon/2 \text{ and }  d_H (\Sigma_{s_t^q}, \Sigma_{s(q_t)}) < \varepsilon/2.
\end{equation*}

We now claim that $\{\theta(p_t), \theta(q_t)\}_{t \in [0,1]}$ is a sweepout of $\mathbb{S}^1$. The embeddedness of $\sigma$ and the fact that it is homotopically non-trivial implies, via the Jordan Theorem, that $\theta: \mathbb{S}^1 \rightarrow \mathbb{S}^1$ is a degree one map. The claim follows from this fact and Proposition \ref{prop-sweepout-characterization}. The maps that appear in the statement are precisely $x_t=\gamma(\theta(p_t))$, $y_t=\gamma(\theta(q_t))$, $s_x(t) = s(p_t)$, and $s_y(t)= s(q_t)$.
\end{proof}

\begin{proof}[Proof of Theorem \ref{main-theorem}]
   Given $\varepsilon>0$, let $\{\Sigma_t\}_{t\in [-1,1]}$ be a sweepout of $(\mathbb{S}^2,g)$ such that $\sup_{t\in [-1,1]}L_g(\Sigma_t)\leq w_B(\mathbb{S}^2,g)+\varepsilon$. By Theorem \ref{thm.compare-Birkhoff}, there exists a sweepout $\{x_t,y_t\}_{t\in [0,1]}$ of $\gamma$ and continuous functions $s_x,s_y:[0,1]\rightarrow [-1,1]$ such that 
    \begin{equation*}
        d_g(x_t,\Sigma_{s_x(t)}),d_g(y_t,\Sigma_{s_y(t)}),d_H(\Sigma_{s_x(t)},\Sigma_{s_y(t)})<\frac{\varepsilon}{6}
    \end{equation*}
    for every $t\in [0,1]$. Picking $z_t\in \Sigma_{s_x(t)}$ closest to $x_t$, $w_t\in \Sigma_{s_y(t)}$ closest to $z_t$, and $u_t\in \Sigma_{s_{y}(t)}$ closest to $y_t$, we have, by the triangle inequality,
    \begin{equation*}
        d_g(x_t,y_t)< \frac{1}{6}\varepsilon + d_H(\Sigma_{s_x(t)},\Sigma_{s_y(t)}) + \frac{1}{2}L_g(\Sigma_{s_y(t)}) + \frac{1}{6}\varepsilon < \varepsilon + \frac{1}{2}w_{B}(\mathbb{S}^2,g)
    \end{equation*}
    for every $t\in [0,1]$. By definition of the min-max width of $\gamma$, 
    \begin{equation*}
        \mathcal{S}_g(\gamma)<\frac{1}{2}w_B(\mathbb{S}^2,g)+\varepsilon.
    \end{equation*}
    Since $\varepsilon>0$ is arbitrary, the proof is finished.
\end{proof}

\begin{exam} \label{example-antipodal-map}
     Let $(\mathbb{S}^2, g)$ be a positively curved sphere for which the antipodal map is an isometry. This metric descends to a metric on the projective plane $\mathbb{RP}^2$. Let $\alpha$ be a simple closed geodesic on this projective plane that has the least length among homotopically non-trivial loops. Then, $\alpha$ lifts to a simple closed geodesic $\gamma$ on $(\mathbb{S}^2, g)$ for which $\mathcal{S}_g(\gamma) = L_g(\gamma)/2$, see Example 9 of \cite{AmbMonSan}. The positivity of the curvature allows the construction of a sweepout of $(\mathbb{S}^2,g)$ by curves, which is admissible for the Birkhoff invariant, such that all curves have length at most $L_g(\gamma)$. (See \cite{Cal-Cao}). Using the estimate on $w_B(\mathbb{S}^2,g)$ obtained from this observation together with Theorem \ref{main-theorem}, we conclude that $\mathcal{S}_g(\gamma) = w_B(\mathbb{S}^2,g)/2$. 
\end{exam}

    \section{Non-compact manifolds} \label{section::non-compact-case}

The proof of Theorem \ref{theorem-non-compact} is based on two theorems of independent interest that we prove in Section \ref{subsec-statements}. In Section \ref{section-aux-lemm} we establish auxiliary lemmas that clarify the connection between these two results with the quantitative topological dimension theory of Urysohn, in the form of estimates for the $1$-width in that theory. In Section \ref{subsec-proof-thm-B} we give a streamlined proof of Theorem \ref{theorem-non-compact}, based on the results of the previous two sections. We conclude with another application, concerning systoles, in Section \ref{subsec-further-res}.

\subsection{Key results}\label{subsec-statements} 

We first relate a uniform upper bound for the widths of all embedded circles to an upper bound on the distance in $M$ to a given ray, or to a given line, on $(M,g)$. Recall that a line is (the trace of) a geodesic $\sigma : \mathbb{R}\rightarrow M$, parametrized by arc-length, such that the restriction of $\sigma$ to each finite interval $[t_0,t_1]$ is the only minimizing geodesic joining $\sigma(t_0)$ and $\sigma(t_1)$. Rays are defined similarly, replacing $\mathbb{R}$ by $[0,\infty)$. 

In what follows, we also use the notation $N_r(\Omega) := \{p \in M : d(p,\Omega) \le r\}$.

\begin{thm}\label{thm-widts-->dist-lines}
Let $(M,g)$ be a complete non-compact Riemannian manifold such that
    \begin{equation*}
        S:= \sup \{ \mathcal{S}_g(\gamma) \,|\, \gamma \text{ is an embedded circle in } M\}<\infty.
    \end{equation*}
    Then, $M$ has at most two ends. Moreover:
\begin{itemize}
            \item[(i)] If $M$ has two ends, then for every line $\sigma \subset M$, we have $N_S(\sigma) =M$. 
            \item[(ii)] If $M$ has one end,  for each ray $\sigma$, there exists a compact set $K = K(\sigma)$ such that $M \backslash K(\sigma) \subset N_S(\sigma)$.
\end{itemize}
\end{thm}
\begin{proof}
    Suppose, by contradiction, that $M$ has at least three ends. Let $B_k$ be a metric ball of radius $k$ in $(M,g)$. If $k$ is large enough, then $M\setminus B_k$ has at least three non-compact connected components. Let $x_k, y_k, z_k \in M$ be points in different non-compact connected components of $M\setminus B_k$. Consider a simple closed curve $\gamma_k$ made of three minimizing geodesic segments joining $x_k$ and $y_k$, $y_k$ and $z_k$, and $z_k$ and $x_k$. We prove that $\lim_{k\rightarrow \infty} S(\tilde{\gamma}_k) = \infty$, for a sequence of embedded circles $\tilde{\gamma}_k$ that are close to $\gamma_k$. 
    
    Indeed, given $k_0$ large, there exist $k_1\geq k_0$ such that, for every $k \geq k_1$, a minimizing geodesic joining $x_k$ and $y_k$ as above does not intersect the connected component of $M\setminus B_{k_1}$ that contains $z_k$. Otherwise, a connected subset of this minimizing segment contained in $M\setminus B_{k_0}$ would have length at least $2(k_1-k_0)$ and endpoints in the closure of $B_{k_0}$. Since these endpoints in the closure of $B_{k_0}$ are at distance at most $2k_0$ between themselves, this is not possible for large $k_1$. Similar observations hold for the other two minimizing segments considered in the construction. 
    
    In particular, following the notation in the statement of Proposition \ref{prop-widthestimate-triangle}, $x_k$, $y_k$, and $z_k$ determine arcs $\gamma_{xy}$, $\gamma_{yz}$, and $\gamma_{zx}$ of $\gamma$ for which 
    \begin{equation*}
        d(x_k, \gamma_{yz}), d(y_k, \gamma_{zx}), d(z_k, \gamma_{xy}) \geq k-k_1.
    \end{equation*}
    Proposition \ref{prop-widthestimate-triangle} yields an embedded circle $\tilde{\gamma}_k$, close to $\gamma_k$, satisfying $S(\tilde{\gamma}_k)> k-k_1-1$. This contradicts the main assumption in the statement, proving that $M$ has at most two ends.

    We consider next the case where $M$ contains a line. Suppose, by contradiction, that there exists a line $\sigma$ in $M$ and $p\in M$, such that $d(p, \sigma) > S$. Up to a reparameterization, we may assume that $d(p, \sigma(0)) = d(p, \sigma)$. Let $T_k$ be a geodesic triangle, with vertices at $p$, $\sigma(-k)$ and $\sigma(k)$, whose sides are minimizing geodesic segments, for each positive integer $k$. 
    
    By the contradiction assumption, the distance between $p$ and the opposite side is bigger than $S$. Let $\alpha_+$ and $\alpha_-$ denote the sides of $T_k$ opposite to $\sigma(k)$ and $\sigma(-k)$, respectively. Notice that, for every $q\in \alpha_-$, 
    \begin{align*}
        d(\sigma(-k), q) \geq  2k - d(q, \sigma(k)) \geq  2k - d(p, \sigma(k)) \geq k - d(\sigma(0),p).
    \end{align*}
     Letting $k$ be large enough, one obtains $d(\sigma(-k), \alpha_-) > S$ and, similarly, $d(\sigma(k), \alpha_+) > S$. Thus, Proposition \ref{prop-widthestimate-triangle} applied to such triangle $T_k$ provides us again with an embedded circle with min-max width bigger than $S$, contrary to the assumptions. This proves item $(i)$.

    In the remainder of the proof, $M$ has a single end, and $\sigma \subset M$ is a ray. Take $L\geq S$ large (to be chosen), and let $K$ be the union of the closure of the geodesic ball $B = B_{L}(\sigma(0))$ and every bounded connected components of $M\setminus \overline{B}$. Notice that the complement of $K$ in $M$ is the only unbounded connect component of $M\setminus \overline{B}$. The compactness of $K$ follows from the fact that a divergent sequence $\{p_n\} \subset K$ would necessarily have a subsequence, not relabeled, not containing two points in the same bounded connected component of $M\setminus \overline{B}$. Using paths from $\sigma(0)$ to $p_n$, we could find $q_n$, in the same bounded connected component of $M\setminus \overline{B}$ that the point $p_n$, and such that $d(q_n, \sigma(0))=2L$. Up to subsequence, $q_n\rightarrow q \in M$. On the other hand, the properties of $\{q_n\}$ imply that $d(q_n, \overline{B}) \geq L$, and $q_n$ and $q_m$ belong to different connected components of $M\setminus \overline{B}$, for $m\neq n$. Then, $d(q_n, q_m) \geq 2L$, contradicting the convergence of this subsequence.
    
    Let us prove that $K$ has the properties stated in $(ii)$, for some sufficiently large $L$. Since $K$ is compact and $\sigma$ is a ray, there exists $k\in \mathbb{N}$ such that $\sigma([k, \infty])$ is contained in the unbounded component of $M\setminus \overline{B}$. By construction of $K$, $\sigma([k,\infty])$ is contained in $M\setminus K$. Suppose, by contradiction, that there exists $p\in M\setminus K$ such that $d(p, \sigma) > S$. Using a path contained in $M\setminus K$ joining $p$ and $\sigma$, we find $q \in M\setminus K$ such that $S< d(q, \sigma) \leq 3S/2$. Notice that, for $t < L - 3S/2$, we have
    \begin{align*}
        d(q, \sigma(t)) \geq d(q, \sigma(0)) - d(\sigma(0),\sigma(t)) > L - t > 3S/2 \geq d(q, \sigma).
    \end{align*}
    In particular, there exists $t_q \geq L-3S/2$ such that $d(\sigma(t_q), q) = d(\sigma, q)$. 
    
    Consider the geodesic triangle $T(L,q)$ whose sides are minimizing geodesic segments joining the points $q$, $\sigma_- = \sigma(t_q-(L-3S/2))$, and $\sigma_+ = \sigma(t_q+(L-3S/2))$. The portion of $\sigma$ between $\sigma_-$ and $\sigma_+$ is the side of the triangle opposite to the point $q$.

    By Proposition \ref{prop-widthestimate-triangle}, it remains to show that, for an appropriated choice of $L>3S/2$, the distance between each vertex of $T(L,q)$ to the opposite side of the triangle is greater than $S$. For the vertex $q$, we know that the distance to the opposite side equals $d(q, \sigma(t_q)) = d(q, \sigma) >S$. For the vertex $\sigma_-$, if $q^{\prime}$ is any point in the opposite side, then
    \begin{align*}
        d(\sigma_-,q^{\prime}) & \geq d(\sigma_-, \sigma_+) - d(q^{\prime}, \sigma_+)\\
        \nonumber &\geq 2L-3S - d(q, \sigma_+)\\
        \nonumber & > 2L-3S - d(q, \sigma(t_q)) - d(\sigma(t_q), \sigma_+)\\
        \nonumber & \geq 2L - 3S - 3S/2 - (L-3S/2) = L-3S.  
    \end{align*}
    The right hand side of the above inequality is bigger than $S$, if $L$ is large. The case of $\sigma_+$ is completely analogous. The proof is complete.
\end{proof}

    We remark that, in the situation described in item $(ii)$ of Theorem \ref{thm-widts-->dist-lines}, $M$ has no lines. In fact, given a line $\sigma:\mathbb{R} \to M$, consider the two rays $\sigma_1,\sigma_2:[0,+\infty) \to M$, defined by $\sigma_1(t)=\sigma(t)$ and $\sigma_2(t) = \sigma(-t)$. For large $t>0$, $\sigma_2$ enters $N_S(\sigma_1)$, so that $t\le d(\sigma_2(t),\sigma_1) \le S $, leading to a contradiction. 
    In the situation described in item $(i)$, on the other hand, the two ends of any line are clearly forced to diverge towards the two distinct ends of $M$.  

    The second result is a consequence of Theorem \ref{prop.algoritmo.v2}. It gives an  estimate for the min-max widths of embedded circles in terms of the diameter of the level sets of continuous real-valued functions. This result is relevant even in the compact setting.
     \begin{thm}\label{thm-diam-->width}
        Let $(M,g)$ be a complete Riemannian manifold. Suppose that there exist a continuous map $f:M \to \mathbb{R}$ and a real number $C>0$ such that $diam (f^{-1}(t)) \le C$ for every $t \in \mathbb{R}$. Then, 
        \begin{equation*}
            \mathcal{S}_g(\gamma) \le C
        \end{equation*}
        for every embedded circle $\gamma$ in $M$.
    \end{thm}
    \begin{proof}
    Let $\gamma$ be an embedded circle in $M$. Let $c:= \min \{f(p) : p\in \gamma\}$ and $d:= \max\{f(p) : p\in\gamma \}$. By virtue of the assumption on the diameter of the level sets of $f$, it is straightforward to check that $f^{-1}([c,d])$ is compact. 
   
    We claim that, for every $\varepsilon>0$, there exists $\delta>0$ such that, if $s_1, s_2 \in [c, d]$ and $|s_1-s_2|<\delta$, then 
    \begin{equation*}
    diam(f^{-1}(s_1)\cup f^{-1}(s_2)) < C+ \varepsilon.
    \end{equation*}
    Suppose, by contradiction, that there exists $\varepsilon>0$ and sequences $\{s_{\ell}\}$ and $\{t_{\ell}\}$ in $[c,d]$ converging to the same limit $\overline{t}$, and such that the diameter of $f^{-1}(s_{\ell})\cup f^{-1}(t_{\ell})$ is at least $C+\varepsilon$. Since $diam(f^{-1}(s_{\ell}))$ and $diam(f^{-1}(t_{\ell}))$ are both bounded from above by $C$, there must exist $x_{\ell} \in f^{-1}(s_{\ell})$ and $y_{\ell} \in f^{-1}(t_{\ell})$ such that $d(x_{\ell}, y_{\ell})\geq C+\varepsilon$. The compactness of $f^{-1}([c,d])$ implies that, up to subsequence, we can suppose that $x_{\ell} \rightarrow p$ and $y_{\ell}\rightarrow q$, with $f(p)=f(q)=\overline{t}$. In particular, for large $\ell$, we have
    \begin{align*}
        d(x_{\ell}, y_{\ell})\leq d(x_{\ell}, p)+d(p,q)+d(q, y_{\ell}) < C+ \varepsilon.
    \end{align*}
    This is a contradiction, and the claim is proved.

    Fix a parametrization of $\gamma : \mathbb{S}^1\rightarrow M$. Fix $\varepsilon>0$, and choose $\delta>0$ given by the claim above. Then there exists a Morse function $h : \mathbb{S}^1\rightarrow \mathbb{R}$, with a single critical point in each critical level, such that $|h-f\circ \gamma |_{C^0}< \delta/2$. Theorem \ref{prop.algoritmo.v2} provides us with a sweepout $\{p_t, q_t\}_{t\in [0,1]}$ of $\mathbb{S}^1$ satisfying $h(p_t) = h(q_t)$. Hence, for every $t\in [0,1]$,
    \begin{equation*}
        |f(\gamma(p_t))- f(\gamma(q_t))|< \delta.
    \end{equation*}
    Thus, we obtained a sweepout $\{\gamma(p_t), \gamma(q_t)\}_{t\in [0,1]}$ of $\gamma$ such that
    \begin{equation*}
    d(\gamma(p_t), \gamma(q_t)) \leq diam(f^{-1}(f(\gamma(p_t)))\cup f^{-1}(f(\gamma(q_t)))) < C+ \varepsilon
    \end{equation*}
    for every $t\in [0,1]$. Therefore $\mathcal{S}_g(\gamma)\leq C+\epsilon$. Since $\varepsilon>0$ is arbitrary, the result follows.
\end{proof}

\subsection{Auxiliary lemmas} \label{section-aux-lemm}

The next lemma concerns the situation where the distance to some line or ray is bounded. Recall that the Busemann function associated to the ray $\sigma: [0, +\infty ) \rightarrow M$ is defined by $B_{\sigma}(p): = \lim_{t\rightarrow +\infty} d(p, \sigma(t)) - t$.
     
    \begin{lemm} \label{prop-busemann-->strong-urysohn}
        Let $(M,g)$ be a complete non-compact Riemannian manifold. If $\sigma$ is a ray in $M$ for which $N_C(\sigma) = M$ for some $C >0$, then the associated Busemann function $B_\sigma : M \to \mathbb{R}$ is such that 
        \begin{equation*}
            \sup\{ \text{diam}(B_\sigma^{-1}(\lambda)) : \lambda \in \mathbb{R} \}\le 4 C.
        \end{equation*}
        If $\sigma$ is a line in $M$ for which $N_C(\sigma)=M$ for some $C>0$, then the same estimate holds for the diameter the level sets of the Busemann function of any ray contained in $\sigma$.
    \end{lemm}

    \begin{proof}
    Given $p\in M$, there exists $t_p\in [0, \infty)$ such that $d(p, \sigma) = d(p, \sigma(t_p))$. This distance is bounded from above by $C$, by assumption. We claim that, for $t>t_p$, we have
    \begin{align*}
     - t_p-d(p, \sigma) \leq d(p, \sigma(t)) - t \leq -t_p + d(p, \sigma). 
    \end{align*}
    Indeed, $d(p, \sigma(t))\leq d(p, \sigma(t_p))+d(\sigma(t_p), \sigma(t)) = d(p, \sigma) + t - t_p$ if $t>t_p$, which implies the second inequality above. The first inequality follows from $t-t_p=d(\sigma(t_p), \sigma(t))\leq d(p, \sigma(t))+d(p, \sigma(t_p))$. For fixed $p$, letting $t$ grow to infinity, one can bound the Busemann function by $- t_p-d(p, \sigma) \leq B_{\sigma}(p) \leq -t_p + d(p, \sigma)$.

    Let $p_1, p_2 \in M$ be such that $B_{\sigma}(p_1) = B_{\sigma}(p_2)$. The inequalities obtained in the previous paragraph for $p_i$, $i=1,2$, give us 
    \begin{equation*}
-t_{p_1}-d(p_1, \sigma) \leq -t_{p_2} + d(p_2, \sigma) \text{ and } -t_{p_2}-d(p_2, \sigma) \leq -t_{p_1} + d(p_1, \sigma).
    \end{equation*}
    Hence, $|t_{p_1}-t_{p_2}|\leq d(p_1, \sigma)+d(p_2, \sigma)\leq 2C$. In particular,
    \begin{equation*}
    d(p_1, p_2) \leq d(p_1, \sigma(t_{p_1}))+d(\sigma(t_{p_1}), \sigma(t_{p_2})) + d(\sigma(t_{p_2}), p_2)\leq 4C. 
    \end{equation*}
\end{proof}

Next, we analyze continuous real functions whose level sets have uniformly bounded diameter.

  \begin{lemm} \label{prop-finite-graph-vs-R}
        Let $(M,g)$ be a complete non-compact Riemannian manifold. The following assertions are equivalent:
        \begin{enumerate}
            \item There exists a continuous function $f: M \to \mathbb{R}$ all of whose level sets $f^{-1}(t)$ have uniformly bounded diameter.
            
            \item The manifold $M$ has at most two ends, and there exist a finite graph $G$ and a continuous function $f: M \to G$ such that 
$$\sup \{ diam f^{-1}(h) : h \in G \} < +\infty.$$
        \end{enumerate}
    \end{lemm}

    \begin{proof}
Let $G$ be a finite graph and $f: M \to G$ be a continuous map such that there exists $C>0$ satisfying $\text{diam}(f^{-1}(h)) \le C$, for every $h \in G$. Fix an exhaustion $K_1\subset K_2 \subset \ldots\subset  K_n\subset \ldots$ of $M$ by compact sets. Every end of $M$ is identified by a sequence of nested unbounded sets $U_1\supset U_2 \supset \ldots \supset U_m \supset \ldots$, where $U_n$ is a connected component of $M\setminus K_n$. Let us begin by explaining the behavior of $f(p)$ as the point $p$ diverges to infinity at the end $N$ determined by $\{U_n\}$. Endow the graph with the metric for which every edge of $G$ has length one, and the distance is the shortest length of paths joining endpoints.  \\

\textbf{Claim:} \begin{enumerate}
\item[(a)] There exists $g=g(N) \in G$ such that for small $\varepsilon>0$, there exists $n_0 \in \mathbb{N}$ and a connected component $I(N, g, \varepsilon)$ of $\overline{B_{\varepsilon}(g)} \setminus \{g\}$ such that $f(U_n)\subset I(N, g, \varepsilon)$, for all $n\geq n_0$. This can be interpreted as: $f$ converges monotonically to a point $g\in G$ as the point of $M$ diverges to infinity in a given end.
\item[(b)] A convergent sequence $p_n \rightarrow p$ in $M$ such that $f(p_n)$ converges to $g=g(N)$ must satisfy $f(p_n) \notin I(N, g, \varepsilon)$, for large $n$ and small $\varepsilon>0$.
\item[(c)] If $N$ and $N^{\prime}$ are different ends with the same $g(N) = g(N^{\prime})$, then, for small $\varepsilon$, one must have $I(N, g, \varepsilon) \neq I(N^{\prime}, g, \varepsilon)$.
\end{enumerate}

\begin{proof}[Proof of the Claim]
Let $\{p_n\}$ be a sequence in $M$ diverging towards the end $N$. For instance, one can take $p_n \in U_n$, for all $n\in \mathbb{N}$. Passing to a subsequence, if necessary, we may assume that $f(p_n)$ converges to some $g \in G$. Since the sequence diverges, the main hypothesis on $f$ easily implies that, for large $n$, $f(p) \neq g$, whenever $p\in U_n$. 

For small $\varepsilon$, $\overline{B_{\varepsilon}(g)} \setminus \{g\}$ is a disjoint union of finitely many connected components which are homeomorphic to line segments open at $g$, and closed at the other endpoint. For large $n$, $f(U_n) \subset \overline{B_{\varepsilon}(g)}\setminus \{g\}$. Indeed, we already know that $g$ does not belong to that image. It remains to verify the inclusion $f(U_n) \subset \overline{B_{\varepsilon}(g)}$. Otherwise, there would be a sequence $\{q_n\}$, such that $q_n\in U_n$ and $f(q_n) \notin \overline{B_{\varepsilon}(g)}$. On the other hand, we already know that the connected set $f(U_n)$ intersects that ball. Hence, $f(U_n)$ would contain one of the finitely many boundary points of $\overline{B_{\varepsilon}(g)}$, for large $n$. Which contradicts the main hypothesis on $f$.

Using the connectedness once more, observe that $f(U_n)$ is contained in one of the connected components of $\overline{B_{\varepsilon}(g)} \setminus \{g\}$, for large $n$. Because the sequence is nested, this connected component does not depend on $n$, and we use $I(N, g, \varepsilon)$ to denote it. This proves the first part of the claim.

Part (b) is proven by contradiction. Since $I(N,g, \varepsilon^{\prime})\subset I(N,g, \varepsilon)$ for $0<\varepsilon^{\prime}<\varepsilon$ sufficiently small, it suffices to show the result for one small value of $\varepsilon$. Suppose, there exists $p_n \rightarrow p$ such that $f(p_n)$ converges to $g$ and $f(p_n) \in I(N, g, \varepsilon)$, for large $n$. Up to subsequence, there would be a divergent sequence $\{q_n\}$ in $N$ such that $f(q_n)$ belongs to the segment of $I(N,g, \varepsilon)$ contained between $f(p_n)$ and $f(p_{n+1})$. The connectedness of $U_n$ would imply on the existence of another divergent sequence $\{x_n\}$ in $N$ such that $f(x_n) = f(p_n)$, contradicting the main hypothesis on $f$. 

The proof of part (c) is similar to the proof of part (b). If two ends had the same behavior $I = I(N, g, \varepsilon) \neq I(N^{\prime}, g, \varepsilon)$, there would exist divergent sequences $\{p_n\}$ in $N$ and $\{p_n^{\prime}\}$ in $N^{\prime}$ such that $f(p_n^{\prime})$ belongs to the segment of $I$ determined by $f(p_n)$ and $f(p_{n+1})$. Then, we would find $\{q_n\}$ divergent in $N$ such that $f(q_n)=f(p_n^{\prime})$, which contradicts the hypothesis on $f$. 
\end{proof}

Next, we prove that (2) implies (1). Let us prove first the special case of the proposition in which the following additional hypothesis holds: $M$ has two ends $N_1$ and $N_2$, and there exist two different vertices $v_1,v_2 \in G$ such that a unique edge is incident to $v_i$, $i=1, 2$, and $f$ converges monotonically in the sense of the claim to $v_1$ on $N_1$, and to $v_2$ on $N_2$. In this case, for small $\varepsilon$, the closure $B$ of the complement of the $\varepsilon$-neighborhood of $\{v_1, v_2\}$ is such that $f^{-1}(B)$ is compact in $M$. In particular, $f^{-1}(B)$ has bounded diameter. Consider now the space $X$ obtained from $G$ by collapsing $B$ to a point, and let $\pi: G \rightarrow X$ denote the corresponding quotient map. Note that $X$ is homeomorphic to $[0,1]$. Define $F = \pi \circ f$ to complete the proof in this special case. Similarly, if $M$ has a single end, and $f$ converges monotonically in the sense of the claim to a vertex at which a unique edge is incident, it is possible to collapse the closure of the complement of an $\varepsilon$-neighborhood of this point and conclude that (1) also holds.

We describe a procedure to modify $f$ and $G$ slightly, by disconnecting edges from a vertex that represents the behavior of $f$ on an end, and connecting it to a new vertex. This reduces the general case to the special case above. Let $N$ be an end of $M$, represented by the nested sequence $\{U_n\}_n$. Suppose that $f$ converges to some $g \in G$ as the point diverges to infinity on $N$, in the sense of the claim. Creating more edges that subdivide edges of $G$, if necessary, we may assume that $g$ is a vertex of $G$. Denote by $h$ the unique vertex of $G$ which is connected to $g$ by an edge $E$ with the property that $f(U_n) \subset E$, for large $n$. Suppose that more edges are incident to $g$. 

Let us construct a new graph $\tilde G$, whose set of vertices is $\tilde V = V \cup \{v\}$, where $V$ is the set of vertices of $G$ and $v$ is a new vertex. Two vertices $u,w  \in \tilde V\setminus \{v, g\}$ are connected in $\tilde G$ if, and only if, they are connected in $G$. The vertices connected to $g$ in $\tilde G$ are exactly those which were connected to this vertex at $G$, with the only exception of $h$. Finally, $v$ is connected to $h$ only. There is a natural bijection between the topological space $\tilde G \backslash \{v\}$ and $G$, sending edges to edges in an affine way, which we denote by $\phi: G \to \tilde G \backslash \{v\}$. This map sends edges of $G$ to edges of $\tilde G$ with the same endpoints, and $E\setminus \{g\}$ to the edge connecting $h$ and $v$. Observe that $\phi \circ f: M \to \tilde G$ is continuous, despite the fact that $\phi$ is discontinuous at $g$. Indeed, the second part of the claim proved above implies that a convergent sequence $p_n \rightarrow p$ in $M$ such that $f(p_n)$ converges to $g$ must satisfy $f(p_n) \notin E\setminus \{g\}$, for large $n$.

Note that $v\in \tilde G$ is a special vertex, in the sense that there is a unique edge incident to it. Moreover, because $\phi$ is a bijection, the levels of $\phi\circ f$ are levels of $f$, which implies that the levels of the new function also have uniformly bounded diameter. In the case in which $M$ has two ends, if necessary, perform the analogous construction associated to the other end, which finally reduces the problem to the special case treated before.

In order to prove that (1) implies (2), it suffices to show that $M$ has at most two ends, as $f: M\rightarrow \mathbb{R}$ composed with a homeomorphism between the real line and $(0,1)$ can be considered as a function taking values on the graph that is composed of a single edge connecting two vertices. Let $h=h(N) \in \mathbb{R}$ be the value associated to an end $N$ of $M$. Assume that $N$ is associated to a nested sequence of connected non-compact subsets $\{U_n\}$. 

Let us prove that $h \notin f(M)$. Part (a) of the claim tells us that, for small $\varepsilon>0$, one would either have $h-\varepsilon <f(p)< h$, for all large $n$ and $p\in U_n$, or $h<f(p)< h+\varepsilon$, for all large $n$ and $p\in U_n$, depending on whether $I(N,h, \varepsilon) = (h-\varepsilon, h)$ or $(h, h+\varepsilon)$. If there exists $q\in M$ such that $f(q)=h$, let $\alpha:[0,1]\rightarrow M$ be a continuous path joining $p\in U_n$ and $q$. Then, $f(\alpha([0,1]))$ covers the whole closed interval joining $f(p)$ and $h$, contained in $I(N,h, \varepsilon)$. In particular, the compact set $\alpha([0,1])$ would contain a convergent sequence $p_n\rightarrow p$ such that $f(p_n)\rightarrow h$ monotonically in $I(N,h, \varepsilon)$. Which would contradict part (b) of the claim. 

Therefore, $h=h(N)$ belongs to the closure of $f(M)$, but is not to $f(M)$. Since $f(M)\subset \mathbb{R}$ is an interval, there are at most two admissible values for $h(N)$, namely its endpoints. For each endpoint, only one case $I(N,h, \varepsilon) = (h-\varepsilon, h)$ or $(h, h+\varepsilon)$ is possible, because $f(M)$ can only intersect one of these intervals. Finally, part (c) implies that two different ends of $M$ must correspond to different intervals $I(N,h, \varepsilon)$. The proof is complete.
\end{proof}

\subsection{Proof of Theorem B} \label{subsec-proof-thm-B}

    Given the preparatory results of the previous sections, we can now give a streamlined proof of Theorem \ref{theorem-non-compact}, which we state again here for the reader's convenience.

\begin{thm} \label{theorem-non-compact-v2}
       Let $(M,g)$ be a complete non-compact Riemannian manifold. The following are equivalent:
        \begin{enumerate}
            \item $\sup \{ \mathcal{S}_g(\gamma) \,|\, \gamma \text{ is an embedded circle in }  M\} < +\infty$,
            \item The manifold $M$ has at most two ends, and there exist a finite graph $G$ and a continuous function $f: M \to G$ such that 
            \begin{equation*}
                \sup_{h \in G} diam f^{-1}(h) < +\infty.
            \end{equation*}
        \end{enumerate}
	\end{thm}                

\begin{proof}
    Assume $(1)$. By Theorem \ref{thm-widts-->dist-lines}, $M$ has at most two ends. If $M$ has two ends, then it necessarily has a line and, in this case, part (i) of Theorem \ref{thm-widts-->dist-lines} and Lemma \ref{prop-busemann-->strong-urysohn} imply that the level sets of the Busemann function of any ray contained in that line have uniformly bounded diameters. If, on the other hand, $M$ has a single end, then part (ii) of Theorem \ref{thm-widts-->dist-lines} applied to any ray $\sigma \subset M$ implies that $M\setminus K \subset N_S(\sigma)$, for some compact set $K$, where $S$ is an upper bound for widths of embedded circles in $(M,g)$. In particular, $M$ is contained in $N_C(\sigma)$, for $C = \max \{S, \max \{d(x, \sigma) : x\in K\}\}$. Therefore, Lemma \ref{prop-busemann-->strong-urysohn} can be applied to the  distance function to the ray $\sigma$, finishing the proof that $(1)$ implies $(2)$. The other implication follows as immediately from the combination of Lemma \ref{prop-finite-graph-vs-R} and Theorem \ref{thm-diam-->width}. 
\end{proof}

\begin{exam} \label{example-simplicial-complex-restriction}
    Let $\mathcal{C}$ denote the standard right circular cylinder $\{(x, y, z) \in \mathbb{R}^3 : x^2+y^2=1\}$. Let $g$ be a Riemannian metric on $\mathbb{S}^1\times \mathbb{R}$ obtained from the metric of $\mathcal{C}$ by replacing it on small geodesic balls centered at the points $(1, 0, k)$, $k\in \mathbb{N}$, in such a way that the modified metric on each of these balls looks like a thin spike whose diameter tends to infinity as $k\rightarrow \infty$. Let $x_k$ denote a point at the tip of the spike associated to the point $(1, 0, k)$. Rounding off the corners of the triangles of vertices $x_k, x_{2k}$, and $x_{3k}$, whose sides are minimizing geodesics, yields a sequence of embedded circles for which the min-max widths tend to infinity. Here, Proposition \ref{prop-widthestimate-triangle} is applied as in the proof of Theorem \ref{thm-widts-->dist-lines}.

    Slightly modified examples of finite area and containing circles of unbounded min-max widths can be obtained by performing the same construction on $\mathbb{S}^1\times \mathbb{R}$ with a metric of finite area. Similar constructions can also be carried out in surfaces with a single end.
\end{exam}

\subsection{Further results} \label{subsec-further-res}
    
    We now establish a result for complete Riemannian manifolds $(M,g)$, compact or not, in which the systole of $(M,g)$ is positive and attained by some homotopically non-trivial curve. Recall that the systole is the least length of homotopically non-trivial curves $c:\mathbb{S}^1\rightarrow M$: 
    \begin{equation*}
        sys(M,g) = \inf\{L_g(c) : [c]\neq 0 \text{ in } \pi_1(M)\}.
    \end{equation*}

       Any homotopically non-trivial curve $\gamma$ attaining $sys(M,g)$ is a closed embedded geodesic. We now compute its min-max width.
       
       \begin{lem}\label{lem:sys}
        Let $(M^n,g)$, $n\geq 2$, be a complete Riemannian manifold with $\pi_1(M)\neq 0$. If $\sigma$ is a homotopically non-trivial curve with length $L_g(\sigma)=sys(M,g)$, then
        \begin{equation*}
            \mathcal{S}_g(\sigma)=\frac{1}{2}sys(M,g).
        \end{equation*}
       \end{lem}
       \begin{proof}
        Pick two different points in $\sigma$, which divide it into two arcs $\sigma_1$, $\sigma_2$ with $L_g(\sigma_1)\leq L_g(\sigma_2)$. If $\sigma_1$ does not realize the distance between its extremities, then there is some other geodesic $c$ joining these points with length $L_g(c)<L_g(\sigma_1)$. Notice that $\sigma_1$ and $c$ are homotopic with fixed extremities, otherwise its concatenation would define a homotopically non-trivial curve in $M$ with length $<L(\sigma)$. But then, the concatenation of $c$ and $\sigma_2$ is homotopic to $\sigma$, and therefore a homotopically non-trivial curve in $M$  with length $<L_g(\sigma)$, a contradiction. 
        Therefore, the distance in $(M,g)$ between any two points of $\sigma$ is realized by the length of $\sigma_1$, the shortest arc of $\sigma$ they define. According to Theorem D in \cite{AmbMonSan}, this condition is equivalent to the equality $\mathcal{S}_g(\sigma)=\frac{1}{2}L_g(\sigma)$, as we wanted to prove.
        \end{proof}

        Combining this information with Theorem \ref{thm-diam-->width}, we obtain:
    
    \begin{cor} \label{corolario-sistoles}
        Let $(M^n,g)$, $n\geq 2$, be a Riemannian manifold with $\pi_1(M) \neq 0$. Suppose there exits a continuous function $f:M\rightarrow \mathbb{R}$ such that, for some constant $C>0$, 
        \begin{equation*}
            \sup \{ diam(f^{-1}(t)) : t\in \mathbb{R}\}\leq C.
        \end{equation*}
        If there exists a homotopically non-trivial curve $\sigma$ with length $sys(M,g)$,  then
        \begin{equation*}
            sys(M,g) \le 2 \cdot C
        \end{equation*}
    \end{cor}
    
    \begin{proof}
        Let $\sigma$ be a homotopically non-trivial curve in $M$ with $L_g(\sigma)=sys(M,g)$. By Lemma \ref{lem:sys}, $\mathcal{S}_g(\sigma) = \frac{1}{2}sys(M,g)$. The inequality follows immediately from the estimate for $\mathcal{S}_g(\sigma)$ from above established in Theorem \ref{thm-diam-->width}.
    \end{proof}
    We remark that the inequality in Corollary \ref{corolario-sistoles} is sharp in the sense that it is attained by the right circular cylinder in $\mathbb{R}^3$ and the orthogonal projection onto its axis, or by the round real projective plane $\mathbb{RP}^2$ and the constant map. Also, by a result quoted in \cite{Bal}, the above inequality holds on compact essential manifolds (see Lemma 2.2.2 therein). The proof of that result relates the systole to the Urysohn $(n-1)$-width of the manifold, which is bounded from above by the supremum of the diameter of the leaves of any map from the manifold to the real line. It gives a better estimate than Corollary \ref{corolario-sistoles} for this class of manifolds. 
    
\section{Rigidity results for Zoll spheres} \label{section::rigidity-results}

 The Birkhoff invariant of any Zoll sphere is the common length of its geodesics. The estimate $diam(\mathbb{S}^2,g)\leq \frac{1}{2} w_{B}(\mathbb{S}^2,g)$ implies that every embedded circle in a Zoll sphere $(\mathbb{S}^2,g)$ satisfies the inequality $\mathcal{S}_g(\gamma)\leq \frac{1}{2}w_B(\mathbb{S}^2,g)$. The content of Theorem \ref{theorem-zoll-rigidity-intro}, which we will prove in the next pages, is that if some geodesic of a Zoll sphere has the maximum min-max width allowed by the latter inequality, than the Zoll sphere is round:
 
\begin{thm}\label{theorem-zoll-rigidity}
    Let $(\mathbb{S}^2,g)$ be a Zoll sphere. If there exists a closed geodesic $\sigma$ of $(\mathbb{S}^2,g)$ such that 
    \begin{equation*}
        \mathcal{S}_g(\sigma)=\frac{1}{2}\cdot w_B(\mathbb{S}^2,g),
    \end{equation*}
    then $(\mathbb{S}^2,g)$ has constant Gaussian curvature. 
\end{thm}

The proof requires some intermediate steps. The first is a well known result about geodesics in Zoll spheres, and the second is a consequence of Theorem \ref{prop.algoritmo.v2}. 

\begin{lemm} \label{optimal-sweepout-lemma}
    Let $(\mathbb{S}^2,g)$ be a Zoll sphere. If $\sigma$ is a closed geodesic of $(\mathbb{S}^2,g)$, then there exists a sweepout $\{\Sigma_t\}_{t \in [-1,1]}$ of $\mathbb{S}^2$ with the following properties:
    \begin{itemize}
        \item[$(i)$] $\Sigma_0 = \sigma$ and there exists a homeomorphism $F:\mathbb{S}^2_0\rightarrow \mathbb{S}^2$ such that $\Sigma_t = F(\overline{c}_t)$ for every $t\in [-1,1]$;
        \item[$(ii)$] $L_g(\Sigma_t) \le L_g(\sigma)$, for every $t \in [-1,1]$, and $L_g(\Sigma_t) = L_g(\sigma)$ if and only if $t=0$; and
        \item[$(iii)$] for every $\varepsilon>0$, there exists $\delta>0$ such that $L_g(\Sigma_t) > w_B(\mathbb{S}^2,g) - \varepsilon$ if and only if $|t|<\delta$.
    \end{itemize}  
\end{lemm}
\begin{proof}[Sketch of proof]
    As $\sigma$ lies within a two-parameter family of geodesics, the nullity of $\sigma$ is at least two, which implies that its index is at least one. Then, for some small $\varepsilon_0>0$, there exists a foliation $\{\Sigma_t\}_{t\in (-\varepsilon_0,\varepsilon_0)}$ around $\sigma=\Sigma_0$, so that the $L_g(\Sigma_t)=L_g(\sigma)-at^2+o(t^2)$ as $t\rightarrow 0$ for some $a>0$. Moreover, the mean curvature vectors of each $\Sigma_t$, $t\neq 0$, points away from $\sigma$. The value of $\varepsilon_0>0$ is chosen so that the function $t \mapsto L_g(\Sigma_t)$ is strictly concave for $t \in (-\varepsilon_0,\varepsilon_0)$ with a maximum point at $t=0$.
    
    Apply the curve shortening flow to each curve $\Sigma_{-\varepsilon_0}$ and $\Sigma_{\varepsilon_0}$. The flow preserves embeddings, the flowing curves lie in $\mathbb{S}^2\setminus \cup_{t\in (-\varepsilon_0,\varepsilon_0)} \Sigma_t$, and all flowing curves have less length than the initial curve. By a theorem of Grayson \cite{Gra}, if the flow does not terminate in finite time at a point, while foliating the respective regions $\mathbb{S}^2$ bounded by its initial strictly convex curve, it accumulates at some closed geodesic. But this is impossible as all geodesics have length $L_g(\sigma)$. Combining these flows and the foliation $\{\Sigma_t\}_{t\in(-\varepsilon_0,\varepsilon_0)}$ in the obvious way, we obtain the desired sweepout of $\mathbb{S}^2$.
    \end{proof}

\begin{lemm} \label{key-lemma-zoll}
    Let $(\mathbb{S}^2,g)$ be a Zoll sphere. If
    \begin{equation*}
        \sup\{ \mathcal{S}_g(\gamma) : \gamma \text{ is an embedded circle in } \mathbb{S}^2\} = \frac{1}{2}w_B(\mathbb{S}^2,g),    
    \end{equation*} 
    then diam$(\mathbb{S}^2,g) = \frac{1}{2}w_B(\mathbb{S}^2,g)$ and, for every closed geodesic $\sigma$, there exist points $p_\sigma$ and $q_\sigma$ in $\sigma$ such that $d(p_\sigma,q_\sigma) = \frac{1}{2}w_B(\mathbb{S}^2,g)$. 
    
    Moreover, if there exists an embedded circle $\alpha$ such that $\mathcal{S}_g(\alpha)=\frac{1}{2}w_B(\mathbb{S}^2,g)$, then the points $p_\sigma$ and $q_\sigma$ can be chosen to lie on $\alpha$.
\end{lemm}

\begin{proof}
     Let $\sigma$ be a closed geodesic of $(\mathbb{S}^2,g)$ and consider the sweepout $\{\Sigma_t\}_{t \in [-1,1]}$ given by Lemma \ref{optimal-sweepout-lemma}, with the corresponding homeomorphism $F$. Given $\varepsilon >0$, let $\delta=\delta(\varepsilon)>0$ be as in Lemma \ref{optimal-sweepout-lemma}. Let $\tau \in (0,\delta)$ be such that $d_H(\Sigma_{t_1},\Sigma_{t_2})< \frac{1}{4}\varepsilon$ if $|t_1-t_2|<\tau$, where $d_H$ denotes the Hausdorff distance.
    
     We define a continuous map $f:\mathbb{S}^2\to \mathbb{R}$ by $f( F(\overline{c}_t(\theta))) = t$, using that $F$ is a homeomorphism. Let $\beta_\varepsilon: \mathbb{S}^1 \to \mathbb{S}^2$ be an embedded circle such that $\mathcal{S}(\beta_\varepsilon) > \frac{1}{2}w_B(\mathbb{S}^2,g) - \frac{1}{4}\varepsilon$. We approximate $f \circ \beta_\varepsilon: \mathbb{S}^1 \to \mathbb{R}$ by a Morse function $\xi$ such that every critical level of $\xi$ contains a unique critical point and $|f\circ \beta_\varepsilon - \xi|_{C^0} < \frac{1}{2}\tau$. By Theorem \ref{prop.algoritmo.v2} there exists a sweepout $\{p_t,q_t\}_{t \in [0,1]}$ of $\mathbb{S}^1$ by pairs of points such that $\xi(p_t) = \xi(q_t)$ for every $t \in [0,1]$. Thus $\{\beta_\varepsilon(p_t),\beta_\varepsilon(q_t)\}_{t \in [0,1]}$ is a sweepout of $\beta_\varepsilon$ by pairs of points and $|f(\beta_\varepsilon(p_t))-f(\beta_\varepsilon(q_t))| < \tau$ for every $t \in [0,1]$. 
     
     Introducing the notation $s_p(t) = f(\beta_\varepsilon(p_t))$ and $s_q(t) = f(\beta_\varepsilon(q_t))$, we compute
     \begin{equation*}
        d(\beta_\varepsilon(p_t),\beta_\varepsilon(q_t)) \le d_H(\Sigma_{s_p(t),},\Sigma_{s_q(t)}) + \frac{1}{2}L_g(\Sigma_{s_q(t)}) < \frac{1}{4}\varepsilon+ \frac{1}{2}L_g(\Sigma_{s_q(t)})
     \end{equation*} 
     
     \noindent for every $t \in [0,1]$.  
    
    Consider $t_* \in [0,1]$, $t_* = t_*(\varepsilon)$, such that $d(\beta_\varepsilon( p_{t_*}), \beta_\varepsilon (q_{t_*})) > \frac{1}{2}w_B(\mathbb{S}^2,g) - \frac{1}{4}\varepsilon$. Then 
    \begin{equation*}
        \frac{1}{2}w_B(\mathbb{S}^2,g) - \frac{1}{4}\varepsilon < \frac{1}{4}\varepsilon+ \frac{1}{2}L_g(\Sigma_{s_q(t_*)}),
    \end{equation*}
    and we conclude that $|s_q(t_*)| < \delta$, because of the properties of the sweepout $\{\Sigma_t\}_{t \in [-1,1]}$ described in Lemma \ref{optimal-sweepout-lemma}. Analogously, $|s_p(t_*)|<\delta$. 
    
    Letting $\varepsilon\to 0^+$, we have $\delta \to 0^+$, and we may assume that $\beta_\varepsilon(p_{t_*(\varepsilon)}) \to p_\sigma$, $\beta_\varepsilon(q_{t_*(\varepsilon)}) \to q_\sigma$. Then $d(p_\sigma,q_\sigma) \ge \frac{1}{2}w_B(\mathbb{S}^2,g)$. Since every Zoll sphere has diameter at most $\frac{1}{2}w_B(\mathbb{S}^2,g)$, we conclude that $diam(\mathbb{S}^2,g) = \frac{1}{2}w_B(\mathbb{S}^2,g)$ and that $d(p_\sigma,q_\sigma) = \frac{1}{2}w_B(\mathbb{S}^2,g)$. 
    
    If there exists an embedded circle $\alpha$ such that $S(\alpha)=\frac{1}{2}w_B(\mathbb{S}^2,g)$, then $\beta_\varepsilon$ in the argument above can be chosen equal to be the curve $\alpha$, and the construction then forces $p_\sigma$ and $q_\sigma$ to lie on $\alpha$. 
\end{proof}

\begin{proof}[Proof of Theorem \ref{theorem-zoll-rigidity}]
    Fix $p \in \sigma$ and let $v \in T_p \mathbb{S}^2$ denote a unit vector that is not tangent to $\sigma$. Let us denote by $\gamma_v$ the closed geodesic that leaves $p$ with velocity $v$. By Lemma \ref{key-lemma-zoll}, $diam(\mathbb{S}^2,g) = \frac{1}{2}w_B(\mathbb{S}^2,g)$ and there exist points $q$ and $\hat q$ in $\gamma_v \cap \sigma$ such that $d(q,\hat q) = \frac{1}{2}w_B(\mathbb{S}^2,g)$. In a Zoll sphere, any two distinct geodesics intersect at exactly two points (\textit{cf.} \cite{Sab}, Theorem 3.3). Thus, we proved that there exists $\hat p \in \mathbb{S}^2$ such that $d(p,\hat p) = \frac{1}{2}w_B(\mathbb{S}^2,g)$ and $\gamma_{v} \cap \sigma = \{p,\hat p\}$. Therefore, each connected component of $\gamma_{v} \backslash \{p,\hat p\}$ has length $\frac{1}{2}w_B(\mathbb{S}^2,g)$. 
    
    Let us denote by $U_1$ and $U_2$ the connected components of the complement of $\sigma$ in $\mathbb{S}^2$. We now prove that the surface with boundary $(\overline{U}_1,g)$ has constant Gaussian curvature. A chord in $(\overline{U}_{1},g)$ is a geodesic segment $\gamma:[0,1] \to \overline{U}_{1}$ such that $\gamma(0) \in \partial \overline{U}_{1}$, $\gamma(1) \in \partial \overline{U}_{1}$ and $\gamma(t) \in \overline{U}_{1} \backslash \partial \overline{U}_{1}$ for every $t \in (0,1)$. In the previous paragraph, we proved that every chord in $(\overline{U}_{1},g)$ has length $\frac{1}{2}w_B(\mathbb{S}^2,g)$. By Bangert's theorem \cite{Ban-chords}, in this situation $(\overline{U}_1,g)$ is isometric to a round hemisphere. Analogously, $(\overline{U}_2,g)$ has constant Gaussian curvature, and this completes the proof. 
\end{proof}

    The following corollary holds for Zoll spheres admitting a circle action by isometries (for instance, the original rotationally symmetric spheres in $\mathbb{R}^3$ that were constructed originally by Zoll \cite{Zol}).

\begin{cor}
    Let $(\mathbb{S}^2,g)$ be a Zoll sphere, and suppose $g$ is intrinsically rotationally symmetric. If 
    \begin{equation*}
        \sup\{ \mathcal{S}_g(\gamma) : \gamma \text{ is an embedded circle in } \mathbb{S}^2\}  = \frac{1}{2}w_B(\mathbb{S}^2,g),
    \end{equation*}
    \noindent then $g$ has constant Gaussian curvature.
\end{cor}
\begin{proof}
    From the classification of intrinsically rotationally symmetric Zoll spheres (see Theorem 4.11 in \cite{Bes}), there exists a closed geodesic $\sigma$ of $(\mathbb{S}^2,g)$ that is invariant under the isometric rotation action, which we denote by $\theta: \mathbb{S}^1 \times \mathbb{S}^2 \to \mathbb{S}^2$. Lemma \ref{key-lemma-zoll} implies that  $diam(\mathbb{S}^2,g) = \frac{1}{2}w_B(\mathbb{S}^2,g)$, and that there exist points $p$ and $\hat p$ in $\sigma$ such that $d(p,\hat p) = \frac{1}{2}w_B(\mathbb{S}^2,g)$. The map $\phi: \sigma \to \sigma$ defined by $\phi(\theta \cdot p) = \theta \cdot \hat p$ is continuous and satisfies $d(x,\phi(x)) = d(p,\hat{p})$ for every $x \in \sigma$. By Theorem C in \cite{AmbMonSan}, $\mathcal{S}_g(\sigma)=\frac{1}{2}w_B(\mathbb{S}^2,g)$ and therefore $g$ has constant Gaussian curvature, by Theorem \ref{theorem-zoll-rigidity}.  
\end{proof}

    We now consider the invariant $W_d(\mathbb{S}^2,g)$. The following result is essentially contained in the proof of Theorem 1.12 in \cite{MonRib}.
    
    \begin{prop}[\textit{cf.} \cite{MonRib}, Theorem 1.12] \label{thm-comparison}
    If $(\mathbb{S}^2,g)$ is a Riemannian sphere, then
    \begin{equation*}
    w_B(\mathbb{S}^2,g) \ge 2 \cdot W_d(\mathbb{S}^2,g).
    \end{equation*}
    
    \end{prop}

\begin{proof}
    Let $F:S^2_1(0) \to \mathbb{S}^2$ be a continuous map of degree one such that $F(\overline{c}_t(\cdot))$ is smooth by parts for every $t \in [-1,1]$. We parametrize the curves $\overline{c}_t:[0,2\pi) \to \mathbb{S}^2_1(0)$ proportionally to their arc-length. Fix $t_0 \in [0,2\pi)$ and consider the continuous maps $F_i : S^2_1(0) \to \mathbb{S}^2$, $i=1,2$, defined by $F_1=F$ and $F_2(\overline{c}_t(s)):= F(\overline{c}_t(t_0))$ for every $t \in [-1,1]$. Then $F_1$ is homotopic to the identity map, by Hopf's degree theorem, and $F_2$ is homotopic to a constant map. Hence $\{F_1(x),F_2(x)\}_{x \in \mathbb{S}^2}$ defines a sweepout by pairs of points of the sphere $(\mathbb{S}^2,g)$, in the sense of \cite{MonRib}, and satisfies:
    \begin{equation}\label{chain-ineq-wd-vs-wb}
        W_d(\mathbb{S}^2,g) \le \max_{x \in \mathbb{S}^2} d_g(F_1(x),F_2(x)) \le \max_{t \in [-1,1]} \frac{1}{2} L_g(F \circ \overline{c}_t).
    \end{equation} 

    Since $F$ is arbitrary, we conclude that $ W_d(\mathbb{S}^2,g) \le \frac{1}{2}w_B(\mathbb{S}^2,g)$. 
\end{proof}

Among Zoll spheres, the equality in Proposition \ref{thm-comparison} characterizes the round metric:

\begin{thm}\label{thm-Zollrigiditywd}
    Let $(\mathbb{S}^2,g)$ be a Zoll sphere. If $w_B(\mathbb{S}^2,g) = 2 \cdot W_d(\mathbb{S}^2,g)$, then $(\mathbb{S}^2,g)$ has constant Gaussian curvature.
\end{thm}

\begin{proof}
    Let $\sigma:\mathbb{R}/\mathbb{Z} \to (\mathbb{S}^2,g)$ be a closed geodesic parametrized proportionally to arc-length. Consider the sweepout constructed from $\sigma$ as in Lemma \ref{optimal-sweepout-lemma} and the corresponding homeomorphism $F$. We construct a sweepout of $\mathbb{S}^2$ by pairs of points as in Proposition \ref{thm-comparison}: we pair the map $F_1=F$ with the map $F_2(\overline{c}_t(s))=F(\overline{c}_t(t_0))$, where $t_0 \in [0,2\pi)$ is fixed. The equalities in the chain of inequalities in \eqref{chain-ineq-wd-vs-wb} for this sweepout then show that 
    \begin{equation*} \label{minimization-prop}
        d(\sigma(t_0),\sigma(t_0+1/2)) = \frac{1}{2}w_B(\mathbb{S}^2,g).
    \end{equation*}
    Since $t_0$ is arbitrary, it follows that $\mathcal{S}_g(\sigma)= \frac{1}{2}w_B(\mathbb{S}^2,g)$, by Theorem C in \cite{AmbMonSan}. The result now follows from Theorem \ref{theorem-zoll-rigidity}.    
\end{proof}

    Theorem \ref{teo-E--Wd-vs-wB} now follows by combining Proposition \ref{thm-comparison} and Theorem \ref{thm-Zollrigiditywd}.

    We remark that the proof of Theorem \ref{thm-Zollrigiditywd} could have ended with an application of L. W. Green's Theorem \cite{Gre}. This is because it was obtained that every piece of geodesic of length not bigger than $\frac{1}{2}w_B(\mathbb{S}^2,g)$ is minimizing. A Zoll sphere with this property is a \textit{wiedersehensfl\"ache}, in the sense of considered by L. W. Green in \cite{Gre}.

\section{Examples} \label{section::Examples}

We collect examples that illustrate different aspects of the theory developed in this paper. The first example concerns the relation between the Almgren-Pitts min-max width and the min-max width of embedded circles.

\begin{exam} \label{example-almgren-pitts}

    The Riemannian spheres $(\mathbb{S}^2,g_n)$ constructed in \cite{Almgren-Starfish-Mantoulidis-Kuo} define a sequence on which the Almgren-Pitts min-max width is constant. But these Riemannian spheres converge to a complete thrice-punctured sphere, for which we know, by Theorem \ref{theorem-non-compact}, that there are embedded circles of arbitrarily large min-max width. Thus, among these examples we can find embedded circles whose min-max width are much larger than the Almgren-Pitts min-max width of the ambient sphere.

\end{exam}

The next example concerns special singular metrics on the sphere that are obtained as limits of sequences of convex spheres (i.e. spheres with positive curvature). In this setting, we study the limit behavior of the geometric invariants $w_B(\mathbb{S}^2,g_n)$, $W_d(\mathbb{S}^2,g_n)$, as well as the min-max width of special embedded circles in $(\mathbb{S}^2,g_n)$, such as their geodesics of least length and other curves that are close to having the largest possible min-max width. This allows us to prove that there exist convex spheres $(\mathbb{S}^2,g)$ where the following holds: there exists an embedded circle $\sigma$ in $(\mathbb{S}^2,g)$ such that $\mathcal{S}_g(\sigma) > W_d(\mathbb{S}^2,g)$, and every geodesic of least length $\gamma$ of $(\mathbb{S}^2,g)$ satisfies $\mathcal{S}_g(\gamma) < W_d(\mathbb{S}^2,g)$. The first inequality illustrates that for such Riemannian spheres the lower bound for $w_B(\mathbb{S}^2,g)$ in terms of the min-max width of embedded circles given by Theorem \ref{main-theorem} is strictly better than the lower bound in terms of $W_d(\mathbb{S}^2,g)$ established in Proposition \ref{thm-comparison}. The second inequality shows that the geodesics of least length do not detect this gap.

The singular metrics on the sphere that we study are generalizations of the so called Calabi-Croke sphere, which is the metric obtained by the gluing of two copies of an equilateral triangle by its edges. This metric is conjectured to asymptotically attain the supremum of the length of a shortest closed geodesic (henceforth called systole) among all unit area Riemannian $2$-spheres.

\begin{exam} \label{example-supS>Wd}

Let $M$ be the $2$-sphere obtained by the gluing of two copies of a planar triangle $ABC$ by its edges. Let $R$ be the radius of the inscribed circle in $ABC$, and let $h$ be the shortest height of the triangle.

There exists a sequence of Riemannian metrics $(\mathbb{S}^2,g_n)$ with positive curvature that approximate $M$. (Section 4 of \cite{ManSch} has an explicit construction of $g_n$ in the case that $ABC$ is equilateral.  Briefly, these metrics are obtained by replacing small balls centered at the vertices by rotationally symmetric smooth metrics, and by slightly perturbing the result to have positive curvature).

 We first show that $\lim_{n \to \infty}w_B(\mathbb{S}^2,g_n) = 2h$.  

 The positivity of the curvature of $(\mathbb{S}^2,g_n)$ implies that $w_B(\mathbb{S}^2,g_n)$ coincides with the length of the shortest closed geodesic of $(\mathbb{S}^2,g_n)$, which is embedded (\textit{cf}. \cite{Cal-Cao}). This length is the systole of $(\mathbb{S}^2,g_n)$, denoted by $sys(\mathbb{S}^2,g_n)$.

Let $\gamma_n$ denote a systole of $(\mathbb{S}^2,g_n)$. Passing to a subsequence, $\gamma_n$ converges to a closed chain of line segments $\gamma$ in $M$. More precisely, there exist a piecewise smooth parametrization $\gamma:[0,1] \to M$ such that every non-smooth point $\gamma(t_j)$ is a vertex of $ABC$, and $\gamma$ restricted to the intervals $[t_{j-1}, t_j]$ parametrize, with constant speed, a line segment of $M \backslash \{A,B,C\}$. 

Clearly, $\gamma$ is not contained in the interior of a copy of $ABC$. We divide in two cases to estimate the length of $\gamma$, $L(\gamma) = \lim_{n \to \infty} w_B(\mathbb{S}^2,g_n)$. If $\gamma$ intersects a vertex, it must contain at least two segments of $M \backslash \{A,B,C\}$ reaching that vertex, each of length at least $h$, thus $ L(\gamma) \geq 2h$ in this case. In the other case, $\gamma$ avoids all vertices. In this setting, there exists a pair of consecutive line segments of $\gamma$ whose common endpoint belongs to the interior of an edge of $ABC$, and the other two endpoints belong to the other two edges of $ABC$.  To see this, note that a sequence of consecutive line segments of $\gamma$  that avoids a given edge of the triangle behave similarly to a geodesic on a cone escaping from the vertex, and thus will not close.  In fact, one can prove that there exist two disjoint pairs of segments in $\gamma$ with the aforementioned property. Elementary geometric arguments show that the length of a consecutive pair of line segments with the aforementioned property is at least $h$. (It can be even shown that if equality holds, then $ABC$ is necessarily isosceles). Thus, we have $L(\gamma) \ge 2h$ in this case as well.

On the other hand, given $\varepsilon>0$, the sweepout of each copy of the triangle by line segments parallel to the shortest height induces a sweepout for which all curves have length bounded by $2h + \varepsilon$ in $(\mathbb{S}^2,g_n)$, for large $n$. In particular, the closed geodesics $\gamma_n$ that realize the Birkhoff invariant of $(\mathbb{S}^2,g_n)$ have their length bounded from above by $2h+\varepsilon$. The arguments of the previous paragraph and the fact that $\varepsilon>0$ is arbitrary show that $\lim_{n \to \infty}w_B(\mathbb{S}^2,g_n) = 2h$

We now study the geometric invariant $W_d(\mathbb{S}^2, g_n)$ and compute the limit $\lim_{n \to \infty} W_d(\mathbb{S}^2,g_n) = 2R$. The number $W_d(\mathbb{S}^2, g_n)$ is realized as the distance between a critical pair $\{P_n,Q_n\} \subset \mathbb{S}^2$ with respect to the distance function induced by $g_n$. According to \cite{MonRib}, there exists either (a) a pair of minimizing geodesics connecting $P_n$ to $Q_n$ which is simultaneously stationary (which forces their union to form a closed geodesic); or (b) there exists a set of three simultaneously stationary geodesics with endpoints $P_n$ and $Q_n$.  See Theorems 1.7 and 1.8 of \cite{MonRib} for details. (Recall that simultaneously stationary condition on a set of geodesics is equivalent to the condition that it is not possible to decrease the lengths of all of them in first order with respect to the flow of the same vector field).  

Let us argue that $\lim_{n\rightarrow \infty} W_d(\mathbb{S}^2,g_n) \ge 2R$. Passing to a subsequence, suppose that $P_n \to P \in M$, $Q_n \to Q \in M$, and that each geodesic of the set of simultaneously stationary geodesics converges to either (I) a line segment of $M$ connecting $P$ and $Q$, possibly hitting exactly one edge of the triangle (but avoiding all vertices), or (II) a line segment in one of the copies of $ABC$, joining a vertex to another point in that triangle. These restrictions are straightforward consequences of the fact that the geodesics in the sequence minimize length with respect to their endpoints. 

If case (I) holds for one limit of minimizing geodesics, $P$ and $Q$ belong to the interior of different copies of $ABC$, and all other limits behave similarly. Moreover, it is not hard to verify that (a) above is not possible in this case, for large $n$.  Let $L$ be the common length of the three limiting line segments joining $P$ and $Q$. Let $Q'$ denote the point in the copy of the triangle that contains $P$ which is the reflected image of $Q$ with respect to any of the edges of $ABC$. Now $P$ and $Q'$ are in a common face of the triangle, and we view this face as a triangle $ABC$ in the plane (not in $M$), in order to facilitate a geometric argument. Notice that $Q'$ is the center of the planar triangle $P_A P_B P_C$, where $P_A$ is the point in the plane obtained as the reflected image of $P$ with respect to the line $BC$, and similarly for $P_B$ and $P_C$. Moreover, $L$ is radius of the circle through $P_A$, $P_B$, and $P_C$. The circle through the midpoints of $PP_A$, $PP_B$, and $PP_C$ has radius $L/2$ and intersects all edges of $ABC$. Therefore, $L/2 \geq R$, and we conclude that $ \lim_{n \to \infty}W_d(\mathbb{S}^2,g_n) = L \ge 2R$ in this case.

We prove that the inequality just obtained also holds assuming that (II) occurs. Suppose first that $Q$ is a vertex. The point $P$ must necessarily be on the boundary of the triangle. If $P$ is another vertex, then $W_d(\mathbb{S}^2,g_n)$ converges to the length of a side of $ABC$, and the desired estimate holds easily. The only remaining case is the configuration on which $P$ is in the interior of the edge opposite to $Q$, but then $QP$ is orthogonal to that edge. As the height of the triangle is at least $2R$, we conclude that $\lim_{n\rightarrow \infty} W_d(\mathbb{S}^2,g_n) \ge 2R$ if (II) occurs.

The desired upper bound for $\lim_{n \to \infty} W_d(\mathbb{S}^2,g_n)$ can be proved by the construction of explicit sweepouts, as in the proof of Theorem 1.8 in \cite{MonRib} (see Proposition 3.2 of that article). We abuse notation and describe a sweepout by pairs of points of the singular space $M$ because it is clear that it induces the desired sweepouts by pairs of points on $(\mathbb{S}^2,g_n)$, for large $n$. Let $I$ and $I'$ denote the centers of the inscribed circles of each copy of $ABC$. Given $P$ in the copy of the triangle that contains $I$, let $P'$ denote the point that corresponds to $P$ in the copy that contains $I'$. For each $P\in AI$, consider the set $s_P$ defined as the union of the line segments joining $P$ to $AB$ and $AC$, which are orthogonal to these edges. Similarly, use $s_{P'}$ for the reflected union of segments. Notice that $d(P, Q)\leq 2R$, for all $Q \in s_P \cup s_{P'}$. Pair points of $s_P \cup s_{P'}$ to the point $P$. Similarly, consider the orthogonal segments of points in $BI$ (respectively $CI$) orthogonal to the edges $BA$ and $BC$ (respectively $CA$ and $CB$), pairing them and the corresponding reflected segments to the point in $BI$ (respectively $CI$). This pairing defines an admissible sweepout for which all distances are bounded by $2R$. This completes the proof of the upper bound, hence $\lim_{n \to \infty} W_d(\mathbb{S}^2,g_n) = 2R$, as we wanted to show.

    We now specialize the study to a concrete family of examples. Let $M$ and $(\mathbb{S}^2,g_n)$ denote the associated spaces constructed above now using the triangle with vertices $A = (-1,0)$, $B=(1+\delta,0)$, $C = (0,\sqrt{3})$, for some small $\delta>0$. This triangle is a small perturbation of an equilateral triangle, but it is not isosceles. Let $\gamma_n$ denotes a systole of $(\mathbb{S}^2,g_n)$. 
    
    We now show that  $\mathcal{S}_{g_n}(\gamma_n) < W_d(\mathbb{S}^2,g_n)$ for large $n$. Since the triangle $ABC$ is not isosceles, one can show with the previous arguments that $\gamma_n$ must subconverge to the union of the segments representing the shortest height of each copy of the triangle, that is, the curve $\gamma$ in $M$ obtained by the gluing of the segments connecting $(0,\sqrt{3})$ to $(0,0)$ in each face. Consider the sweepout $\{p_t, q_t\}_{t \in [0,1]}$ of $\gamma$ where the points $p_t$ and $q_t$, on different faces of the triangle, are equidistant along $\gamma$ to the point $(0,0)$. Notice that $d_M(p_t,q_t) < 2R$, for all $t \in [0,1]$. The claim follows from the fact that $\lim_{n \to \infty }W_d(\mathbb{S}^2,g_n) = 2R$.

 Finally, we argue that if $n$ is large enough, there exists an embedded circle $\sigma_n$ in $(\mathbb{S}^2,g_n)$ such that $\mathcal{S}_{g_n}(\sigma_n) > W_d(\mathbb{S}^2,g_n)$. To this end, let $\sigma$ be the curve in $M$ constructed as the boundary of the triangle. We claim that there exists an involution $\phi: \sigma \to \sigma$ such that $d_M(p,\phi(p)) > 2R$ for every $p \in \sigma$. This will be enough because of the convergence of the metrics and due to Lemma \ref{lem-nofixedpts-AMS}. The existence of the involution is obtained as in the proof of Proposition \ref{prop-widthestimate-triangle}, because the distance between each vertex of the triangle and the opposite edge is at least $h>2R$ in $M$. 

 To conclude the discussion of this class of examples originating from these triangles, we highlight two inequalities. If $\gamma_n$ denotes a geodesic of least length in $(\mathbb{S}^2,g_n)$, $\sigma_n$ denotes a sequence of embedded circles converging to $\sigma$, and $n$ is large enough, then we have proved that
 \begin{equation*}
 \mathcal{S}_{g_n}(\gamma_n) < W_d(\mathbb{S}^2,g_n) < \mathcal{S}_{g_n}(\sigma_n) \le \frac{1}{2}w_B(\mathbb{S}^2,g_n),
\end{equation*}
and these were the properties we were looking for. 
\end{exam}

The final example presents a class of metrics where equality holds both in Theorem \ref{main-theorem} and Proposition \ref{thm-comparison}.
\begin{exam}
    Let $(\mathbb{S}^2,g)$ be a Riemannian sphere of positive curvature. Suppose there exists a critical pair of points at the level $W_d(\mathbb{S}^2,g)$ for which there exists a simultaneously stationary collection of minimizing geodesics with two elements. (By Theorem 1.8 in \cite{MonRib} there exists a collection with either two or three elements). In this case, the two minimizing geodesics of the collection form together a closed embedded geodesic $\gamma$. The positivity of the curvature allows us to construct a sweepout of $\mathbb{S}^2$ which has $\gamma$ as curve of largest length (see \cite{Cal-Cao}), so that equality holds in Proposition \ref{thm-comparison}. Arguing as in the proof of Theorem \ref{thm-Zollrigiditywd} one shows that, when parametrized proportionally to arc-length, $\gamma: \mathbb{R}/\mathbb{Z} \to \mathbb{S}^2$ satisfies $d_g(\gamma(s),\gamma(s+1/2)) = \frac{1}{2}w_B(\mathbb{S}^2,g)$, for all $s \in \mathbb{R}$, where $L_g(\gamma) = w_B(\mathbb{S}^2,g)$. Thus $S_g(\gamma) = L_g(\gamma)/2 = W_d(\mathbb{S}^2,g)$ by Theorem D in \cite{AmbMonSan}. In conclusion, $S_g(\gamma) = \frac{1}{2}w_B(\mathbb{S}^2,g) = W_d(\mathbb{S}^2,g)$.

\end{exam}

\bibliographystyle{plain}
\bibliography{refs}

@book {AmbMon,
    AUTHOR = {Ambrozio, Lucas and Montezuma, Rafael},
     TITLE = {M\'etodos min-max em geometria---uma introdu\c c\~ao},
    SERIES = {35$\sp {\rm o}$ Col\'oquio Brasileiro de Matem\'atica},
    VOLUME = {1},
 PUBLISHER = {Instituto Nacional de Matem\'atica Pura e Aplicada (IMPA), Rio
              de Janeiro},
      YEAR = {2025},
     PAGES = {xiv+185},
      ISBN = {978-85-244-0590-7},
   MRCLASS = {58E10 (49Q05 49Q20 53C22 58E05)},
  MRNUMBER = {4993149},
}

@inproceedings {Guth,
    AUTHOR = {Guth, Larry},
     TITLE = {Metaphors in systolic geometry},
 BOOKTITLE = {Proceedings of the {I}nternational {C}ongress of
              {M}athematicians. {V}olume {II}},
     PAGES = {745--768},
 PUBLISHER = {Hindustan Book Agency, New Delhi},
      YEAR = {2010},
      ISBN = {978-81-85931-08-3; 978-981-4324-32-8; 981-4324-32-9},
   MRCLASS = {53C23},
}

@article {AmbMonSan,
    AUTHOR = {Ambrozio, Lucas and Montezuma, Rafael and Santos, Roney},
     TITLE = {The width of embedded circles},
   JOURNAL = {J. Reine Angew. Math.},
  FJOURNAL = {Journal f\"ur die Reine und Angewandte Mathematik. [Crelle's
              Journal]},
    VOLUME = {824},
      YEAR = {2025},
     PAGES = {289--332},
      ISSN = {0075-4102,1435-5345},
   MRCLASS = {58E10 (49Q20 53A04 55M30)},
  MRNUMBER = {4926948},
MRREVIEWER = {Rossella\ Bartolo},
       DOI = {10.1515/crelle-2025-0026},
       URL = {https://doi.org/10.1515/crelle-2025-0026},
}

@article {BalSab,
    AUTHOR = {Balacheff, Florent and Sabourau, Stéphane},
     TITLE = {Diastolic and isoperimetric inequalities on surfaces},
   JOURNAL = {Annales Scientifiques de l'École Normale Supérieure},
  FJOURNAL = {},
    VOLUME = {43},
      YEAR = {2010},
     PAGES = {579--605},
      note = {}
}

@article {Almgren-Starfish-Mantoulidis-Kuo,
    AUTHOR = {Mantoulidis, Christos and Marx-Kuo, Jared},
     TITLE = {Almgren's Three-Legged Starfish},
   JOURNAL = {},
  FJOURNAL = {},
    VOLUME = {},
      YEAR = {},
     PAGES = {},
      note = {arXiv:2504.20895}
}

@book {Bal,
    AUTHOR = {Balitskiy, Alexey},
     TITLE = {Bounds on {U}rysohn {W}idth},
      NOTE = {Thesis (Ph.D.)--Massachusetts Institute of Technology},
 PUBLISHER = {ProQuest LLC, Ann Arbor, MI},
      YEAR = {2021},
     PAGES = {(no paging)},
   MRCLASS = {99-05},
  MRNUMBER = {4464413},
       URL =
              {http://gateway.proquest.com/openurl?url_ver=Z39.88-2004&rft_val_fmt=info:ofi/fmt:kev:mtx:dissertation&res_dat=xri:pqm&rft_dat=xri:pqdiss:29291639},
}

@article {BalCroKat,
    AUTHOR = {Balacheff, Florent and Croke, Christopher and Katz, Mikhail
              G.},
     TITLE = {A {Z}oll counterexample to a geodesic length conjecture},
   JOURNAL = {Geom. Funct. Anal.},
  FJOURNAL = {Geometric and Functional Analysis},
    VOLUME = {19},
      YEAR = {2009},
    NUMBER = {1},
     PAGES = {1--10},
      ISSN = {1016-443X,1420-8970},
   MRCLASS = {53C23 (53C22)},
  MRNUMBER = {2507217},
MRREVIEWER = {Luofei\ Liu},
       DOI = {10.1007/s00039-009-0708-9},
       URL = {https://doi.org/10.1007/s00039-009-0708-9},
}

@article {Ban-chords,
    AUTHOR = {Bangert, Victor},
     TITLE = {Manifolds with Geodesic Chords of Constant Length},
   JOURNAL = {Math. Ann.},
  FJOURNAL = {},
    VOLUME = {265},
    NUMBER = {},
      YEAR = {1983},
     PAGES = {273--281},
      note = {}
}

@book {Bes,
    AUTHOR = {Besse, Arthur L.},
     TITLE = {Manifolds all of whose geodesics are closed},
    SERIES = {Ergebnisse der Mathematik und ihrer Grenzgebiete [Results in
              Mathematics and Related Areas]},
    VOLUME = {93},
      NOTE = {With appendices by D. B. A. Epstein, J.-P. Bourguignon, L.
              B\'erard-Bergery, M. Berger and J. L. Kazdan},
 PUBLISHER = {Springer-Verlag, Berlin-New York},
      YEAR = {1978},
     PAGES = {ix+262},
      ISBN = {3-540-08158-5},
   MRCLASS = {53C20 (53C22 58G99)},
  MRNUMBER = {496885},
MRREVIEWER = {R.\ L.\ Bishop},
}

@article {Bir,
    AUTHOR = {Birkhoff, George D.},
     TITLE = {Dynamical systems with two degrees of freedom},
   JOURNAL = {Trans. Amer. Math. Soc.},
  FJOURNAL = {Transactions of the American Mathematical Society},
    VOLUME = {18},
      YEAR = {1917},
    NUMBER = {2},
     PAGES = {199--300},
      ISSN = {0002-9947,1088-6850},
   MRCLASS = {70H03 (34C25 37C27)},
  MRNUMBER = {1501070},
       DOI = {10.2307/1988861},
       URL = {https://doi.org/10.2307/1988861},
}

@article {Cal-Cao,
    AUTHOR = {Eugenio Calabi and Jian Guo Cao},
     TITLE = {Simple closed geodesics on convex surfaces},
   JOURNAL = {J. Differential Geom.},
  FJOURNAL = {},
    VOLUME = {36(3)},
      YEAR = {1992},
     PAGES = {517-549},
      note = {DOI: 10.4310/jdg/1214453180}
}

@article {ChaLio,
    AUTHOR = {Chambers, Gregory R. and Liokumovich, Yevgeny},
     TITLE = {Optimal sweepouts of a {R}iemannian 2-sphere},
   JOURNAL = {J. Eur. Math. Soc. (JEMS)},
  FJOURNAL = {Journal of the European Mathematical Society (JEMS)},
    VOLUME = {21},
      YEAR = {2019},
    NUMBER = {5},
     PAGES = {1361--1377},
      ISSN = {1435-9855,1435-9863},
   MRCLASS = {53C23 (53C22)},
  MRNUMBER = {3941494},
MRREVIEWER = {Luis\ Guijarro},
       DOI = {10.4171/JEMS/863},
       URL = {https://doi.org/10.4171/JEMS/863},
}

@article {Che,
    AUTHOR = {Cheeger, Jeff},
     TITLE = {A Lower Bound for the Smallest Eigenvalue of the Laplacian},
   JOURNAL = {Problems in Analysis: A Symposium in Honor of Salomon Bochner. Princeton Univ. Press},
  FJOURNAL = {},
    VOLUME = {},
      YEAR = {1970},
    NUMBER = {},
     PAGES = {195--199},
      ISSN = {},
   MRCLASS = {},
  MRNUMBER = {},
MRREVIEWER = {},
       URL = {},
}

@article {Cro,
    AUTHOR = {Croke, Christopher B.},
     TITLE = {Area and the length of the shortest closed geodesic},
   JOURNAL = {J. Differential Geom.},
  FJOURNAL = {Journal of Differential Geometry},
    VOLUME = {27},
      YEAR = {1988},
    NUMBER = {1},
     PAGES = {1--21},
      ISSN = {0022-040X,1945-743X},
   MRCLASS = {53C22 (58E10)},
  MRNUMBER = {918453},
MRREVIEWER = {Gudlaugur\ Thorbergsson},
       URL = {http://projecteuclid.org/euclid.jdg/1214441646},
}

@article {Gra,
    AUTHOR = {Grayson, Matthew A.},
     TITLE = {Shortening embedded curves},
   JOURNAL = {Ann. of Math. (2)},
  FJOURNAL = {Annals of Mathematics. Second Series},
    VOLUME = {129},
      YEAR = {1989},
    NUMBER = {1},
     PAGES = {71--111},
      ISSN = {0003-486X,1939-8980},
   MRCLASS = {53C22 (58E10)},
  MRNUMBER = {979601},
MRREVIEWER = {Gudlaugur\ Thorbergsson},
       DOI = {10.2307/1971486},
       URL = {https://doi.org/10.2307/1971486},
}

@article {Gre,
    AUTHOR = {Green, L. W.},
     TITLE = {Auf {W}iedersehensfl\"achen},
   JOURNAL = {Ann. of Math. (2)},
  FJOURNAL = {Annals of Mathematics. Second Series},
    VOLUME = {78},
      YEAR = {1963},
     PAGES = {289--299},
      ISSN = {0003-486X},
   MRCLASS = {53.72},
  MRNUMBER = {155271},
MRREVIEWER = {P.\ Hartman},
       DOI = {10.2307/1970344},
       URL = {https://doi.org/10.2307/1970344},
}

@article {GroShi,
    AUTHOR = {Grove, Karsten and Shiohama, Katsuhiro},
     TITLE = {A generalized sphere theorem},
   JOURNAL = {Ann. of Math. (2)},
  FJOURNAL = {Annals of Mathematics. Second Series},
    VOLUME = {106},
      YEAR = {1977},
    NUMBER = {2},
     PAGES = {201--211},
      ISSN = {0003-486X},
   MRCLASS = {53C20},
  MRNUMBER = {500705},
MRREVIEWER = {R.\ L.\ Bishop},
       DOI = {10.2307/1971164},
       URL = {https://doi.org/10.2307/1971164},
}

@article{Gui,
 author = {Guillemin, Victor},
 title = {The {Radon} transform on {Zoll} surfaces},
 fjournal = {Advances in Mathematics},
 journal = {Adv. Math.},
 issn = {0001-8708},
 volume = {22},
 pages = {85--119},
 year = {1976},
 language = {English},
 doi = {10.1016/0001-8708(76)90139-0},
 zbMATH = {3549806},
 Zbl = {0353.53027}
}

@article {ManSch,
    AUTHOR = {Mantoulidis, Christos and Schoen, Richard},
     TITLE = {On the {B}artnik mass of apparent horizons},
   JOURNAL = {Classical Quantum Gravity},
  FJOURNAL = {Classical and Quantum Gravity},
    VOLUME = {32},
      YEAR = {2015},
    NUMBER = {20},
     PAGES = {205002, 16},
      ISSN = {0264-9381,1361-6382},
   MRCLASS = {83C57},
  MRNUMBER = {3406373},
MRREVIEWER = {Theophanes\ Grammenos},
       DOI = {10.1088/0264-9381/32/20/205002},
       URL = {https://doi.org/10.1088/0264-9381/32/20/205002},
}

@book {MarMonOli,
    AUTHOR = {Martini, Horst and Montejano, Luis and Oliveros, D\'eborah},
     TITLE = {Bodies of constant width},
      NOTE = {An introduction to convex geometry with applications},
 PUBLISHER = {Birkh\"auser/Springer, Cham},
      YEAR = {2019},
     PAGES = {xi+486},
      ISBN = {978-3-030-03866-3; 978-3-030-03868-7},
   MRCLASS = {52Axx (33C55 52-02 52C17 53C65)},
  MRNUMBER = {3930585},
MRREVIEWER = {Alina\ Stancu},
       DOI = {10.1007/978-3-030-03868-7},
       URL = {https://doi.org/10.1007/978-3-030-03868-7},
}

@article {MonRib,
    AUTHOR = {Montezuma, Rafael and Ribeiro, Idalina},
     TITLE = {The min-max width of spheres associated to the distance function},
   JOURNAL = {},
  FJOURNAL = {},
    VOLUME = {},
      YEAR = {},
     PAGES = {},
      note = {arXiv:2504.20895}
}

@article {Pitts,
    AUTHOR = {Pitts, Jon},
     TITLE = {Regularity and singularity of one dimensional stationary integral varifolds on
manifolds arising from variational methods in the large.},
   JOURNAL = {Symposia Mathematica},
  FJOURNAL = {},
    VOLUME = {Vol. XIV},
      YEAR = {1973},
     PAGES = {465–472},
      note = {arXiv:2504.20895}
}

@article{Rot,
 author = {Rotman, R.},
 title = {The length of a shortest closed geodesic and the area of a {{\( 2\)}}-dimensional sphere},
 fjournal = {Proceedings of the American Mathematical Society},
 journal = {Proc. Am. Math. Soc.},
 issn = {0002-9939},
 volume = {134},
 number = {10},
 pages = {3041--3047},
 year = {2006},
 language = {English},
 doi = {10.1090/S0002-9939-06-08297-9},
}

@article {Sab,
    AUTHOR = {Sabourau, St\'ephane},
     TITLE = {Strong deformation retraction of the space of {Z}oll {F}insler
              projective planes},
   JOURNAL = {J. Symplectic Geom.},
  FJOURNAL = {The Journal of Symplectic Geometry},
    VOLUME = {17},
      YEAR = {2019},
    NUMBER = {2},
     PAGES = {443--476},
      ISSN = {1527-5256,1540-2347},
   MRCLASS = {53C60 (53D25)},
  MRNUMBER = {3992722},
MRREVIEWER = {Vasile\ Sorin\ Sab\u au},
       DOI = {10.4310/JSG.2019.v17.n2.a5},
       URL = {https://doi.org/10.4310/JSG.2019.v17.n2.a5},
}

@article {Tre,
    AUTHOR = {Treibergs, Andrejs},
     TITLE = {Estimates of volume by the length of shortest closed geodesics on a
convex hypersurface},
   JOURNAL = {Invent. Math.},
  FJOURNAL = {},
    VOLUME = {80},
      YEAR = {1985},
     PAGES = {481--488},
      ISSN = {},
   MRCLASS = {},
  MRNUMBER = {},
MRREVIEWER = {},
       URL = {},
}

@article {Wei,
    AUTHOR = {Weinstein, Alan},
     TITLE = {On the volume of manifolds all of whose geodesics are closed},
   JOURNAL = {J. Differential Geometry},
  FJOURNAL = {Journal of Differential Geometry},
    VOLUME = {9},
      YEAR = {1974},
     PAGES = {513--517},
      ISSN = {0022-040X,1945-743X},
   MRCLASS = {53C20},
  MRNUMBER = {390968},
MRREVIEWER = {H.\ Osborn},
       URL = {http://projecteuclid.org/euclid.jdg/1214432547},
}

@article {Zol,
    AUTHOR = {Zoll, Otto},
     TITLE = {Ueber {F}l\"achen mit {S}charen geschlossener geod\"atischer
              {L}inien},
   JOURNAL = {Math. Ann.},
  FJOURNAL = {Mathematische Annalen},
    VOLUME = {57},
      YEAR = {1903},
    NUMBER = {1},
     PAGES = {108--133},
      ISSN = {0025-5831,1432-1807},
   MRCLASS = {99-04},
  MRNUMBER = {1511201},
       DOI = {10.1007/BF01449019},
       URL = {https://doi.org/10.1007/BF01449019},
}

\end{document}